\documentclass[11pt,reqno]{amsart}
\usepackage{a4wide}
\usepackage{amsmath,amsfonts,amssymb,bbm}
\usepackage{verbatim}
\usepackage[shortlabels]{enumitem}
\usepackage[T1]{fontenc}
\usepackage[unicode,hypertexnames=false,colorlinks=true,linkcolor=blue,citecolor=blue]{hyperref}
\usepackage{esint}
\usepackage{hyperref}
\usepackage[capitalise, nameinlink, noabbrev]{cleveref} 
\usepackage{mathrsfs}

\hypersetup{colorlinks,breaklinks,
	linkcolor=[rgb]{0,0,0},
	citecolor=[rgb]{0,0,0},
	urlcolor=[rgb]{0,0,0}}

\usepackage{xfrac,xcolor,color,graphicx}
\newcommand{\be}{\begin{equation}}
	\newcommand{\ee}{\end{equation}}
\theoremstyle{plain} 

\newtheorem{theorem}{Theorem}[section]
\newtheorem{proposition}[theorem]{Proposition}

\newtheorem{lemma}[theorem]{Lemma}
\newtheorem{definition}[theorem]{Definition}

\newtheorem{remark}[theorem]{Remark}

\newcommand{\res}{\mathop{\hbox{\vrule height 7pt width .5pt depth 0pt
			\vrule height .5pt width 6pt depth 0pt}}\nolimits}

\newcommand{\R}{\mathbb{R}}

\newcommand{\N}{\mathbb{N}}

\newcommand{\eps}{\varepsilon}

\newcommand{\HH}{\mathcal{H}}

\DeclareMathOperator{\dive}{div}
\newcommand{\bea}{\begin{equation*}\begin{aligned}}
		\newcommand{\eea}{\end{aligned}\end{equation*}}

\crefname{subsection}{subsection}{subsections}

\numberwithin{equation}{section}

\numberwithin{equation}{section}

\title{Uniqueness of one-phase cones with isolated singularity}

\author[M. Carducci]{Matteo Carducci}\thanks{}
\address {Matteo Carducci \newline \indent
	Classe di Scienze, Scuola Normale Superiore \newline \indent
	Piazza dei Cavalieri 7, 56126 Pisa - ITALY}
\email{\href{mailto:matteo.carducci@sns.it}{matteo.carducci@sns.it}}
\author[B. Velichkov]{Bozhidar Velichkov}\thanks{}
	\address {Bozhidar Velichkov \newline \indent
		Dipartimento di Matematica, Universit\`a di Pisa \newline \indent
		Largo B. Pontecorvo 5, 56127 Pisa - ITALY}
	\email{\href{mailto:bozhidar.velichkov@unipi.it}{bozhidar.velichkov@unipi.it}}
\begin{document}
	
	\subjclass[2010] {
	}
	
	\keywords{Regularity, free boundary, variational solutions, \L ojasiewicz inequality, epiperimetric inequality}
	\subjclass{35R35}
    \begin{abstract}
        We prove uniqueness of the blow-up at every singular point of the one-phase free boundary problem for which one blow-up has an isolated singularity. 
        The result applies to Lipschitz stationary solutions and does not require any integrability assumption on the cone. In this sense, our result completes the picture for uniqueness of tangent cones at isolated singularities in the one-phase problem.
    
        Inspired by Simon's work in the minimal surface setting, we prove an infinite dimensional \L ojasiewicz inequality for the spherical Weiss' energy. This yields an epiperimetric inequality for non-minimizing solutions, which leads to the uniqueness result.
    \end{abstract}
	\maketitle
	\tableofcontents
\section{Introduction}
Free boundary problems arise naturally in analysis, geometry and mathematical physics, and concern the regularity of interfaces whose structure is itself part of the unknown. One of the most studied free boundary problems is the \emph{one-phase problem}.
In its simplest form, one looks for non-negative functions $u:B_1\to\R$ satisfying the overdetermined problem
\be\label{eq:one-phase-problem}\Delta u=0\quad \text{in }\Omega_u\cap B_1,\qquad
     |\nabla u|=1\quad \text{on }\partial \Omega_u\cap B_1,\ee where $\Omega_u:=\{u>0\}$. 
     Since both the function $u$ and its positivity set $\Omega_u$ are unknown, the boundary $\partial\Omega_u\cap B_1$ is called \emph{free boundary}.

Originally motivated by models in fluid mechanics and combustion theory, the one-phase problem has also revealed close connections with the theory of minimal surfaces, shape optimization, harmonic measures and semilinear elliptic equations \cite{Desilva-minimal,traizet,jerisonkamburov,mazzoleni-terracini-velichkov,kenigtoro,fernandez-ros-semilinear}.

A classical way to construct solutions of \eqref{eq:one-phase-problem} is to minimize the Alt-Caffarelli energy functional
\be\label{eq:functional}J(u):=\int_{B_1}|\nabla u|^2\,dx+|\{u>0\}\cap B_1|,\ee among non-negative functions with prescribed boundary values. 
The functional \eqref{eq:functional} was originally introduced in the seminal work \cite{AltCaffarelli:OnePhaseFreeBd} by Alt and Caffarelli. Since then, regularity theory for minimizers has been developed in several directions.
These include regularity outside the singular set \cite{AltCaffarelli:OnePhaseFreeBd,caffarelli1,caffarelli2,DeSilva:FreeBdRegularityOnePhase}, low-dimensional rigidity and construction of singular cones \cite{CaffarelliJerisonKenig04:NoSingularCones3D,DeSilvaJerison09:SingularConesIn7D,JerisonSavin15:NoSingularCones4D}, rectifiability of the singular set \cite{edelen-englestein},
uniqueness of blow-ups at isolated singularities \cite{esv}, foliations around singular cones \cite{DJS,esv23} 
and generic regularity results \cite{generic-alt-phillips,fernandezyu-linearized}.
See also the monographs \cite{CaffarelliSalsa:GeomApproachToFreeBoundary, Velichkov:RegularityOnePhaseFreeBd} for an introduction to the problem.

More recently, regularity theory for several classes of non-minimizing solutions has received increasing attention \cite{weiss2003,jerisonkamburov,graphical,KamburovWang,kriventsov-weiss,ERZ25,ChanFernandezFigalliSerra,rosoton-pasarin,FernandezSerra,guidojonasnahon}.
Such non-minimizing solutions naturally appear in problems involving domain variations \cite{hhp11,lww}, fluid mechanics \cite{bsg76,cg11}, saddle points of the energy \cite{basultokamburov,FernandezFloritSerra}, Perron solutions \cite{Caffarelli3,jerisonkamburov19,feldman}, flame propagation \cite{caffarellivazquez}, and through limits of classical solutions and singular perturbation procedures \cite{weiss2003,kriventsov-weiss}. 
Nevertheless, their regularity theory is much less developed. 

\medskip

In this paper we consider \emph{variational solutions} of the one-phase problem \eqref{eq:one-phase-problem} (see \cref{def:variational-sol}). Roughly speaking, they are non-negative Lipschitz functions, which are harmonic in their positivity set, and stationary for the Alt-Caffarelli functional \eqref{eq:functional} under inner variations.

This class of solutions is very large, and contains minimizers, as well as weak solutions in the sense of \cite{AltCaffarelli:OnePhaseFreeBd} (see \cref{subsec:examples}). However, differently from the latter two classes, the variational solutions do not satisfy, in general, the usual non-degeneracy assumption or density estimates, which makes their analysis more involved. Their compactness and regularity properties were recently studied by Kriventsov and Weiss in \cite{kriventsov-weiss}.

Variational solutions satisfy the Weiss' monotonicity formula \cite{weiss98}. 
As a consequence, given a variational solution $u\in H^1(B_1)$ with free boundary point $0\in\partial\Omega_u$, the rescalings $u_r(x):=r^{-1}u(rx)$ converge, up to subsequences, to a $1$-homogeneous global variational solution $b$. 
Such a limit is called a \emph{blow-up} of $u$ at the origin. A priori, different sequences of scales may lead to different blow-up limits. 
The problem of uniqueness of the blow-up, namely the fact that $b$ is independent of the subsequence, is a central question in geometric analysis and free boundary problems; indeed, it allows to deduce asymptotic behavior of solutions near the singular point and an asymptotic description of the free boundary from a classification of tangent cones. 

\subsection{Main results}

For variational solutions of the one-phase problem, at regular free boundary points, namely if one blow-up is of the form $(x_d)_+$ up to a rotation, the blow-up is unique. This follows from the classical improvement of flatness theorem \cite{DeSilva:FreeBdRegularityOnePhase}, once we know that variational solutions are viscosity solutions in a neighborhood of regular points, and that the free boundaries of the rescalings become flatter and flatter, converging to the free boundary of $(x_d)_+$ (see \cite[Section 8.1]{kriventsov-weiss} or \cref{lemma:existence-blow-up}). 

Regarding the singular set, the situation is much more involved.
For minimizers of the Alt-Caffarelli functional \eqref{eq:functional}, when one blow-up is a cone with an isolated singularity (see \cref{def:cone-isol-sing}), uniqueness was proved by Engelstein, Spolaor and the second named author in \cite{esv} (see also \cite{EdSV}). More recently, Engelstein, Restrepo and Zhao \cite{ERZ25} established a similar result for the weak solutions of the one-phase problem, in the sense of \cite{AltCaffarelli:OnePhaseFreeBd}, under the additional assumption that the limiting cone is integrable through rotations. 

Our main result establishes uniqueness of the blow-up at every cone with isolated singularity, for the more general class of variational solutions of the one-phase problem.

 \begin{theorem}[Uniqueness of the blow-up]\label{thm:uniqueness}
    Let $u\in H^1(B_1)$ be a variational solution of the one-phase problem with Lipschitz constant $L>0$, and let $b$ be a one-phase cone with isolated singularity. Suppose that $b$ is a blow-up of $u$ at $0\in\partial\Omega_u$, then $b$ is the unique blow-up and there exists $r_0>0$ such that $$\|u_r-b\|_{L^\infty(B_1)}\le C|\log r|^{-\alpha}\quad\text{for every }r\in(0,r_0),$$ for some $C>0$ depending only on $d$, $b$, $L$ and some $\alpha\in(0,1)$ depending only on $d$, $b$.
   
    Moreover, the free boundary $\partial\Omega_u\cap B_{r_0}$ is a $C^{1,\log}$ graph over $\partial\Omega_b\cap B_{r_0}$.
    
    Finally, if $b$ is integrable (see \cref{def:integrability}), we can replace $|\log r|^{-\alpha}$ and $C^{1,\log}$ with $r^\alpha$ and $C^{1,\alpha}$ respectively.
\end{theorem}
\smallskip

Let us compare \cref{thm:uniqueness} with the previous uniqueness results \cite{esv,ERZ25}. 
On the one hand, \cite{esv} treats arbitrary cones with isolated singularities, including non-integrable ones, but only in the minimizing setting. In general, the convergence is logarithmic, while a polynomial rate is obtained under an additional integrability assumption (different from the one considered in this paper; see also the discussion below).

On the other hand, \cite{ERZ25} applies to a class of non-minimizing solutions, but requires that the cone $b$ is integrable through rotations, as well as a non-degeneracy property of the solution. In this case, the convergence to the blow-up is polynomial.

In \cref{thm:uniqueness}, we extend the uniqueness results in \cite{esv,ERZ25} to the larger class of variational solutions, in particular without requiring non-degeneracy of the solution $u$, or the integrability of the cone $b$. We prove uniqueness of the blow-up for every cone with isolated singularity, with a logarithmic rate of convergence. If $b$ is assumed to be integrable (or integrable through rotations), we can improve the convergence rate to a polynomial one.

We point out that our notion of integrability, introduced in \cref{def:integrability}, is the natural analogue of the classical notion from minimal surface theory. Roughly speaking, we require that every $1$-homogeneous Jacobi field arises as an
infinitesimal generator of families of one-phase cones. When this family is given by rotations of the cone, we recover the integrability through rotations considered in \cite{ERZ25}. Notice that this notion of integrability through rotations is different from the one considered in \cite{esv}, and is closer in spirit to the linearized equation (see \cite{CaffarelliJerisonKenig04:NoSingularCones3D,JerisonSavin15:NoSingularCones4D,fernandezyu-linearized}). For more details, we refer to the discussion in \cref{rem:differences-integrability}. 

\subsubsection{Examples of cones with isolated singularity} 
Several examples of $1$-homogeneous solutions with an isolated singularity are known in the literature. Caffarelli, Jerison and Kenig \cite{CaffarelliJerisonKenig04:NoSingularCones3D} introduced examples of non-trivial axially symmetric cones, which are not minimizing in dimension $d\le 6$. 
In dimension $d=7$, De Silva and Jerison \cite{DeSilvaJerison09:SingularConesIn7D} proved that the corresponding cone is minimizing; moreover, as shown in \cite{esv,ERZ25}, it is integrable through rotations. Later, Hong \cite{Hong} constructed Lawson-type singular cones whose spherical free boundaries are products of spheres; the stability and the minimality of these cones were studied recently by Firester, Tsiamis and Wang \cite{ftw1,ftw2}. 

In \cite{wang-hong}, Wang and Hong constructed a one-phase cone in dimension
five whose spherical free boundary consists of regular level set of a Cartan-M\"{u}zner isoparametric cubic polynomial. More generally, Wang \cite{wang-cartan} constructed three
families of one-phase cones associated with arbitrary isoparametric foliations of the sphere.
These isoparametric families were recently revisited by Firester, Tsiamis and Zhao \cite{ftz}, who prove uniqueness within each prescribed foliation and study their geometry and
densities. The stability and minimality of some of these cones were investigated in \cite{carduccivelichkovcartan}.
Finally, Hines, Kolesar and McGrath \cite{lavoroa3} constructed new classes of cones in dimension three and four, which are invariant with respect to the action of discrete groups. 

It is currently unknown whether the cones constructed in \cite{Hong,wang-hong,wang-cartan,lavoroa3} are integrable. More generally, no classification of singular one-phase cones is currently available, and it is not
known whether all cones with isolated singularity are integrable. In particular, \cref{thm:uniqueness} establishes the uniqueness of the tangent cones without imposing any structural assumptions, such as stability or integrability, on the singular cone. 

\subsubsection{Examples of variational solutions}\label{subsec:examples}
The uniqueness result in \cref{thm:uniqueness} applies to the very large class of variational solutions, in the sense of \cref{def:variational-sol}.  
Hereafter we discuss some examples of variational solutions for which \cref{thm:uniqueness} applies.

First, weak solutions in the sense of \cite{AltCaffarelli:OnePhaseFreeBd} (or in the sense of \cite{ERZ25}) are variational solutions.
Indeed, if $u$ is such a weak solution, then $\HH^{d-1}(\partial\Omega_u\setminus\partial^\ast\Omega_u)=0$ (see \cite[Lemma 5.3]{AltCaffarelli:OnePhaseFreeBd}) and the reduced free boundary $\partial^\ast\Omega_u$ is locally smooth (see \cite[Theorem 8.4]{AltCaffarelli:OnePhaseFreeBd}). Moreover, since $\Delta u=\HH^{d-1}\res \partial^\ast\Omega_u,$ the free boundary condition $|\nabla u|=1$ holds classically on $\partial^\ast\Omega_u$. Arguing as in the proof of \cite[Theorem 5.1]{weiss98} (see also \cite[Proposition 5.2]{rosoton-pasarin}), one can integrate by parts away from a covering of $\partial\Omega_u\setminus\partial^\ast\Omega_u$, and let the size of the covering tend to zero, thereby obtaining that $u$ is stationary. Since $u$ is also Lipschitz, then $u$ is a variational solution.

Another class of variational solutions is given by Perron solutions of the one-phase problem, in the sense of \cite{Caffarelli3}. 
They are viscosity solutions in the sense of \cref{def:viscosity-sol} and are constructed as an infimum over a suitable class of supersolutions.
 For more details, we refer to \cite{Caffarelli3,CaffarelliSalsa:GeomApproachToFreeBoundary,DESILVA-perron,cv24-fully}.
    The Perron solutions in the sense of \cite{Caffarelli3} satisfy additional properties that are not available for general viscosity solutions. 
In particular, they are Lipschitz continuous and non-degenerate, their free boundaries have locally finite $\HH^{d-1}$-measure, $\partial^\ast\Omega_u$ is locally given by a smooth function (in particular, $\partial_\nu u=-1\quad\text{on }\partial^\ast\Omega_u$), and $\HH^{d-1}(\partial\Omega_u\setminus\partial^\ast\Omega_u)=0.$ Then, by \cite[Theorem 4.5]{AltCaffarelli:OnePhaseFreeBd}, we have $\Delta u=\HH^{d-1}\res \partial^\ast\Omega_u$. This implies that every Perron solution $u$ in the sense of \cite{Caffarelli3} is a weak solution in the sense of \cite{AltCaffarelli:OnePhaseFreeBd}, and thus is a variational solution.
More recently, in \cite{feldman}, a more general family of Perron' solutions was verified to be variational.

\subsection{Main ideas of the proof}

Before giving the main ideas of the proof of the uniqueness result in \cref{thm:uniqueness}, 
we briefly discuss the main differences with the results on the uniqueness of the blow-ups proved in \cite{esv} and \cite{ERZ25}.
The argument in \cite{esv} is based on a logarithmic epiperimetric inequality for the Weiss' energy \bea W(u):=\int_{B_1}|\nabla u|^2\,dx-\int_{\partial B_1}u^2\,d\HH^{d-1}+|\{u>0\}\cap B_1|.\eea Roughly speaking, such inequality provides a competitor whose energy is quantitatively smaller than the $1$-homogeneous function with the same boundary value. 
The uniqueness of the blow-up is deduced via the decay of the Weiss' energy, which is in turn obtained by testing the minimality of the solution with the constructed competitor.
This last step is no longer available for non-minimizing solutions, and makes the approach from \cite{esv} not applicable in our case.

The approach in \cite{ERZ25} is instead inspired by the work of Allard and Almgren \cite{AA81} on minimal surfaces. It does not require any minimality property, and therefore applies to non-minimizing solutions. The argument, however, uses in an essential way the assumption that every Jacobi field is generated by rotations, and thus does not directly apply to non-integrable cones.

In the theory of minimal surfaces, the corresponding integrability assumption was removed in the seminal work of Simon \cite{Simon}. For cones with isolated singularity, Simon introduced an infinite dimensional version of the \L ojasiewicz inequality, and used it to control the evolution of the spherical slices along the radial variable, obtaining uniqueness of the tangent cone also in the non-integrable case. This strategy has subsequently been adapted and developed in several other geometric settings \cite{yangmills,schulze,coldingminicozzi2,coldingminicozzi1,chodoshschulze,edelenminter}.

Inspired by Simon's work, we aim to implement the same strategy in the one-phase setting. Precisely, in order to treat non-integrable cones, we prove a \L ojasiewicz-type inequality for the spherical Weiss' energy $$\mathcal{F}(\phi):=\int_{\partial B_1}\Big(|\nabla_\theta \phi|^2-(d-1)\phi^2\Big)\,d\HH^{d-1}+\HH^{d-1}(\{\phi>0\}\cap\partial B_1).$$ In our framework, however, the functional $\mathcal{F}$ is not analytic and, in fact, is not even differentiable. 
This creates a first fundamental difficulty: the classical \L ojasiewicz inequality cannot even be formulated in terms of the usual gradient of $\mathcal{F}$.
Consequently, we introduce an appropriate notion of gradient, which we call \emph{reduced gradient} and denote by $\nabla_{\text{red}}\mathcal{F}$ (see \cref{subsection-important}). The resulting \L ojasiewicz inequality takes the form (see \cref{thm:loj-ineq})
\be\label{eq:loj-intro}|\mathcal{F}(\phi)-\mathcal{F}(b)|^{1-\theta}\le C\|\nabla_{\text{red}}\mathcal{F}(\phi)\|,\ee for some $\theta\in(0,\frac12]$ and $C>0$. The key point in the proof of such inequality is to decompose the problem into coercive directions and a finite dimensional reduced functional, which can then be treated by the classical finite dimensional \L ojasiewicz inequality \cite{lojasiewicz}. 

When $b$ is integrable, we obtain the strongest form of the inequality \eqref{eq:loj-intro}, corresponding to $\theta=\frac12$. 
Indeed, integrability is equivalent to the vanishing of the reduced functional, in analogy with the result of Adams and Simon \cite{adamsimon}. However, in our setting, establishing this equivalence requires a finer analysis. In particular, we prove a one-to-one correspondence between $1$-homogeneous Jacobi fields and elements of the kernel of the second variation of a parametrization of $\mathcal{F}$. We refer to \cref{subsection:correspondence} for further details.

Another difficulty is that, differently from the minimal surface framework, a \L ojasiewicz inequality alone is not sufficient to obtain uniqueness of the blow-up in the one-phase problem. 
Indeed, after the change of variables $t=-\log r$, the minimal surface equation can be viewed as a nonlinear evolution PDE with gradient-type structure. Simon's argument combines this structure with the \L ojasiewicz inequality in order to obtain uniqueness and rate of convergence.

For the one-phase problem, no analogous direct gradient-type structure is available. Indeed, the reduced gradient of $\mathcal F$ along the spherical slices cannot be directly identified with the radial derivative of the solution. This follows from the structure of the Euler-Lagrange equations of the functional \eqref{eq:functional}: the solution satisfies both an equation in the interior of its positivity set and a condition on the moving free boundary.  

To overcome this difficulty, we prove an additional estimate which controls the reduced gradient in an annulus, in terms of the radial derivative of the solution in a bigger annulus. More precisely, setting $\phi_r(\theta):=r^{-1} u(r,\theta)\in H^1(\partial B_1)$, we prove the following radial control of the reduced gradient (see \cref{prop:propfundamental})
\bea\label{eq:radial-intro}\int_{5/8}^{7/8}\frac1r\|\nabla_{{\rm red}}\mathcal{F}(\phi_r)\|^2\,dr\le C\int_{1/2}^{1}\frac1r\int_{\partial B_1}(\nabla u_r\cdot x-u_r)^2\,d\HH^{d-1}\,dr.\eea
This estimate recovers the link between the reduced gradient and the radial evolution that is needed to implement Simon's argument in our setting.
The key ingredients here are a smooth parametrization lemma (see \cref{lemma:smooth-parametrization-lemma}) and a Caccioppoli-type inequality for moving free boundaries (see \cref{lemma:estimateintime}). In particular, 
the compactness results of \cite{kriventsov-weiss} play a crucial role in the proof of the former result.

Combining the \L ojasiewicz inequality and the radial control of the reduced gradient, with the Weiss' monotonicity formula, we obtain a direct decay estimate for the Weiss' energy, which takes the form of the following epiperimetric-type inequality for variational solutions.

\begin{theorem}[Epiperimetric inequality]\label{thm:epiperimetric-inequality}
    There exist $\eps>0$ and $\delta>0$ depending only on $d$, $b$ and $L$ such that the following holds.
    
    Let $u\in H^1(B_1)$ be a variational solution of the one-phase problem with Lipschitz constant $L>0$, with $0\in\partial\Omega_u$, and let $b$ be a one-phase cone with isolated singularity.
    We suppose that, for some $\rho\in(0,1]$,
    \bea\label{eq:closeness-ass}\|u_{\rho}-b\|_{L^2(\partial B_1)}\le \delta\quad\text{and}\quad |W(u_{\rho})-W(b)|\le \delta.\eea 
    We also assume that $\Theta(u,0):=\lim_{r\to0^+}W(u_r)\ge W(b)$.
    Then, there exists $\gamma\in[0,1)$ depending only on $d$ and $b$, such that
    \be\label{thm:epiperimetric-inequality-eq}
    E(\rho/2)\le (1-\eps E(\rho/2)^\gamma)E(\rho),\quad\text{where}\quad E(\rho):=W(u_\rho)-W(b).
    \ee 
    Moreover, if $b$ is integrable (see \cref{def:integrability}), then we can take $\gamma=0$.
\end{theorem}
Epiperimetric inequalities are used as a powerful tool to establish regularity results 
in several geometric variational problems and free boundaries \cite{reifenberg,white,taylor1,taylor2,Weiss-improvement,SpolaorVelichkov2019:EpiperimetricBernoulli2D,csv18,csv20,esv,cv24,carduccitortoneuniqueness}. This method has been developed for minimizers, since it allows to deduce decay of the Weiss' energy by testing the minimality property. 
However, since the variational solutions considered here are not assumed to be minimizing, the decay \eqref{thm:epiperimetric-inequality-eq} is not obtained by competitors and energy minimality, but it is proved directly from the \L ojasiewicz inequality and the radial control of the reduced gradient.

Direct decay of the Weiss' energy has been obtained in the setting of parabolic obstacle-type problems \cite{shi18,csv-crelle,col2025,carduccicolombotorres}; however, these results rely on the gradient flow structure associated with the parabolic evolution. 
No such structure is present in our elliptic setting, and the proof requires substantially different arguments. In this sense, our result extends the epiperimetric inequality approach to non-minimizing elliptic problems, even without relying on an underlying gradient flow.

We expect that similar strategies can be used in other geometric and free boundary problems, where neither minimality nor a direct gradient flow structure is available.

\subsection{Structure of the paper} The paper is organized as follows. In \cref{sec2}, we define the notion of variational solution, recalling some regularity and compactness properties from \cite{kriventsov-weiss}. In \cref{sec2.5}, we introduce some tools from \cite{esv}, such as the functional $\mathcal{G}$, the second variation, its kernel, and the Lyapunov-Schmidt reduction. In \cref{sec:int}, we define the integrability condition, and we show that it is equivalent to the vanishing of the reduced functional, in the spirit of \cite{adamsimon}. In \cref{sec3}, we define the reduced gradient of $\mathcal{F}$, and we prove the corresponding \L ojasiewicz inequality for $\mathcal{F}$. In \cref{sec4}, we prove that the reduced gradient can be estimated by the radial derivative of the solution. Finally, in \cref{sec6}, we conclude the proof of the epiperimetric inequality in \cref{thm:epiperimetric-inequality}, and of the uniqueness result in \cref{thm:uniqueness}.
\subsection{Use of AI} No AI tools were used to produce mathematical content; the strategy and the proofs were developed and written exclusively by the authors. AI tools were used only for language editing and a final review of the paper.
\subsection{Acknowledgements}
M.C. would like to thank Xavier Fernández-Real for a useful discussion about the non-degeneracy condition in the one-phase problem.
The authors are supported by the European Research Council (ERC) via the project ERC FiRM - {\em Fine structure and regularity of stationary and moving free boundaries} (grant agreement No. 101230705).

\section{Variational solutions}\label{sec2}
\subsection{Solutions of the one-phase problem}
Let $B\subset \R^d$ be a ball and consider $u\in H^1(B)$. We define the Alt-Caffarelli functional $$J(u,B):=\int_{B}|\nabla u|^2\,dx+|\{u>0\}\cap B|.$$

We use the following definition of stationary solutions.
\begin{definition}[Stationary solutions]\label{def:stationary-sol}
    We say that $u\in H^1(B)$ is a stationary solution of the one-phase problem in $B$ if $u\in C^0(B)\cap C^2_{\text{loc}}(\Omega_u\cap B)$, $u\ge0$ in $B$, and
    \bea\label{eq:stationary-cond}\frac{d}{dt}J(u\circ \Psi_t^{-1},B)\Big|_{t=0}=\int_{B}\Big(|\nabla u|^2\dive \xi+\chi_{\{u>0\}}\dive \xi-2\nabla u\cdot D\xi\nabla u\Big)\,dx=0,\eea for every $\xi\in C^1_c(B,\R^d)$, where $\Psi_t(x):=x+t\xi(x).$
\end{definition}
We point out that the stationarity condition is given only for inner variations. The hypothesis $u\in C^2_{\text{loc}}(\Omega_u\cap B)$ is used to ensure that $\Delta u=0$ in $\Omega_u\cap B$. Indeed, integrating by parts the stationarity condition away from the free boundary, we obtain that $\Delta u \nabla u=0$ in $\Omega_u\cap B$. Since $u$ is regular there, then $\Delta u=0$ in $\Omega_u\cap B$. 

\medskip

Throughout the paper, we consider variational solutions in the following sense.
\begin{definition}[Variational solutions]\label{def:variational-sol}
    We say that $u\in H^1(B)$ is a variational solution of the one-phase problem (with Lipschitz constant $L>0$) in $B$ if $u$ is a stationary solution in the sense of \cref{def:stationary-sol}, and $u$ is $L$-Lipschitz in $B$, namely \be\label{eq:lipschitz-regularity}\|\nabla u\|_{L^\infty(B)}\le L.\ee

    We say that $u\in H^1_{\text{loc}}(\R^d)$ is a global variational solution of the one-phase problem (with Lipschitz constant $L>0$), if $u$ is a variational solution (with Lipschitz constant $L$) in every ball $B\subset \R^d$.
\end{definition}

\begin{remark}\label{rem:compare-variational-solutions}
    We emphasize the differences between our notion of variational solution in \cref{def:variational-sol} and the one considered in \cite[Subsection 3.2]{kriventsov-weiss}. The latter is defined for a pair $(u,\chi)$, where $\chi:B\to \{0,1\}$ is an additional variable which generalizes the characteristic function $\chi_{\{u>0\}}$.
This is useful in the study of limits of more regular solutions and compactness arguments, since the characteristic functions of the positivity sets could be not stable under convergence.

However, \cite[Theorem 1.2]{kriventsov-weiss} shows that if $(u,\chi)$ is a variational solution in their sense, then either $u\equiv0$ or $\chi\equiv\chi_{\{u>0\}}$ (the case $u\equiv0$ cannot be excluded, see \cite{weiss2003}). Thus, for nontrivial solutions, the variable $\chi$ is completely determined by $u$. Since in our compactness arguments we only deal with nontrivial limits (see also the discussion in \cref{rem:compactness} or directly \cref{lemma:smooth-parametrization-lemma}), our definition of variational solution is directly given in terms of $u$.
    \end{remark}

In what follows, we give the definition of classical solutions of the one-phase problem.
\begin{definition}[Classical solutions]\label{def:classical-solution}
    Let $A\subset B$ be a domain. We say that $u\in H^1(B)$ is a classical solution of the one-phase problem in $A$ if $u\ge0$, $\partial\Omega_u\cap A$ is $C^{2,\alpha}$ regular, $u\in C^{2,\alpha}(\overline\Omega_u\cap A)$ and \bea\label{eq:sol-classical}
        \Delta u=0\quad \text{in }\Omega_u\cap A,\qquad
        |\nabla u|=1\quad \text{on }\partial \Omega_u\cap A,\eea in pointwise sense.

Let $A\subset \R^d$ be a domain. We say that $u\in H^1_{\text{loc}}(\R^d)$ is a classical solution of the one-phase problem in $A$, if $u$ is a classical solution in $A\cap B$, for every ball $B\subset \R^d$.
    
\end{definition}

Finally, we also use the following definition of viscosity solutions.

\begin{definition}[Viscosity solutions]\label{def:viscosity-sol} We say that $u\in C^0(B)$ is a viscosity solution of the one-phase problem if $u\ge0$ in $B$, $\Delta u=0$ in $\Omega_u\cap B$ and $u$ satisfies the free boundary condition in the following sense. For every point $x_0\in \partial\Omega_u\cap B$ and every test function $\varphi\in C^\infty(B)$ such that $\varphi(x_0)=0$ and $\varphi\le u$ (resp. $\varphi_+\ge u$) in $B_r(x_0)\subset B$ for some $r>0$, we have $|\nabla\varphi(x_0)|\le 1$ (resp. $|\nabla\varphi(x_0)|\ge 1$).
\end{definition}

\subsection{Monotonicity formula for the Weiss' energy}\label{subsection-weiss}
For any function $u\in H^1(B_1)$, we define the Weiss' energy \bea W(u):=\int_{B_1}|\nabla u|^2\,dx-\int_{\partial B_1}u^2\,d\HH^{d-1}+|\{u>0\}\cap B_1|.\eea
We consider the rescalings $u_{r,x_0}(x):=r^{-1}u(x_0+rx)$. When $x_0=0$, we will drop the dependence on $x_0$, writing $u_r$ for $u_{r,0}$.

\medskip

In the following proposition, we recall the  Weiss' monotonicity formula for stationary solutions from \cite[Theorem 3.1]{weiss98}. Since variational solutions are stationary, the monotonicity formula holds for them as well.
\begin{proposition}\label{prop:monotonicity-formula}
    Let $u\in H^1(B_1)$ be a stationary solution of the one-phase problem, and let $x_0\in B_1$. Then, for almost every $r\in(0,\text{dist}(x_0,\partial B_1))$, we have \be\label{eq:monotonicity-formula}\frac{d}{dr} W(u_{r,x_0})=\frac2r\int_{\partial B_1}(\nabla u_{r,x_0}\cdot x-u_{r,x_0})^2\,d\HH^{d-1}.
\ee In particular, $r\mapsto W(u_{r,x_0})$ is non-decreasing, and it is constant if and only if $u$ is $1$-homogeneous with respect to $x_0$.
\end{proposition}
Let us recall the proof of the monotonicity formula in \cite{weiss98}, which is based on the following two facts (see also \cite[Section 9]{Velichkov:RegularityOnePhaseFreeBd}). Assuming without loss of generality that $x_0=0$, we first observe that for any function $u\in H^1(B_1)$, we have the Weiss’ formula
\bea\label{eq:weiss-formula}\frac{d}{dr} W(u_r)=\frac{d}{r}(W(z_r)-W(u_r))+\frac{1}{r}\int_{\partial B_1}(\nabla u_r\cdot x-u_r)^2\,d\HH^{d-1},\eea where $z_r$ is the $1$-homogeneous extension of the trace $u_r|_{\partial B_1}$. 
Second, if $u\in H^1(B_1)$ is a stationary solution, we have the equipartition of the energy
\be\label{eq:equipartition-energy}W(z_r)-W(u_r)=\frac1d\int_{\partial B_1}(\nabla u_r\cdot x-u_r)^2\,d\HH^{d-1}.
\ee
Combining the previous two identities, we conclude \eqref{eq:monotonicity-formula}.

\smallskip

As a consequence of \eqref{eq:monotonicity-formula}, we also get the following useful estimate \be\label{eq:inequality-with-log}\|u_{r_2}-u_{r_1}\|_{L^2(\partial B_1)}\le C \log\left(\frac{r_2}{r_1}\right)^{1/2}(W(u_{r_2})-W(u_{r_1}))^{1/2}\quad\text{for every } 0<r_1\le r_2\le 1.\ee
This can be proved by integrating \eqref{eq:monotonicity-formula} and by applying the H\"older inequality.

\subsection{Properties of variational solutions}
In the following proposition, we recall some of the main results of \cite{kriventsov-weiss} about variational solutions. 
\begin{proposition}\label{prop:kriventsov-weiss}
    Let $u\in H^1(B_1)$ be a variational solution of the one-phase problem, then the following holds.
    \begin{itemize}
        \item[(i)] For every $x_0\in\partial\Omega_u\cap B_1$, the function $\Theta(u,x_0):=\lim_{r\to0^+}W(u_{r,x_0})$ is well-defined and satisfies $\Theta(u,x_0)\in [\omega_d/2,\omega_d]$. 
        Moreover, the map $\partial\Omega_u\cap B_1\ni x_0\mapsto \Theta(u,x_0)$ is upper semicontinuous.
        \item[(ii)] The set $\partial\Omega_u$ is countably $\HH^{d-1}$-rectifiable, with locally finite $\HH^{d-1}$-measure. Moreover, denoting by $\partial^\ast\Omega_u\subset \partial \Omega_u$ the reduced free boundary, and considering \be\label{eq:sigmaH}\Sigma^{H}_u:=\{x_0\in\partial\Omega_u\cap B_1:\ \Theta(u,x_0)=\omega_d\},\ee we have, in the sense of distributions in $B_1$, $$\Delta u=\HH^{d-1}\res \partial^\ast\Omega_u+2c_dH(u,x)^{1/2}\HH^{d-1}\res \Sigma^H_u,$$ where $H(u,x):=\lim_{r\to0^+}r^{-2}\fint_{\partial B_r(x)}u^2\,d\HH^{d-1}$ and $c_d>0$ is a normalization constant.
        \item[(iii)] If $\Sigma^{H}_u=\emptyset$, then $u$ is a viscosity solution in $B_1$, in the sense of \cref{def:viscosity-sol}.
    \end{itemize}
\end{proposition}
\begin{proof}
    First, $\Theta(u,x_0)$ is well-defined by \cref{prop:monotonicity-formula}. Then, (i) follows from \cite[Proposition 3.5]{kriventsov-weiss} and \cite[Lemma 8.1]{kriventsov-weiss}. The point (ii) is exactly \cite[Theorem 8.7]{kriventsov-weiss}, while (iii) follows by \cite[Lemma 8.3]{kriventsov-weiss}. 
    \end{proof}
For the sake of completeness, we recall that, for $\HH^{d-1}$-a.e.~$x_0\in\Sigma^{H}_u$, the blow-up sequence $u_{r,x_0}$ converges, up to a rotation, to $c_dH(u,x_0)^{1/2}|x_d|$. Notice that this includes the possibility that $H(u,x_0)=0$, in which case the blow-up is identically zero and $x_0$ is a degenerate point.
\subsection{Compactness for variational solutions}
In the following proposition, we give a compactness result for variational solutions established in \cite{kriventsov-weiss}.
\begin{proposition}\label{lemma:compactness-kw}
    Let $\{u^j\}\subset H^1(B_1)$ be a sequence of variational solutions of the one-phase problem with the same Lipschitz constant $L>0$, such that $0\in\partial\Omega_{u_j}$ for every $j\in\N$. Then, up to extracting a subsequence, 
    $$u^j\to u^\infty\quad\text{strongly in }H^1(B_{1})\cap C^{0,\alpha}(\overline B_1),$$ for some variational solution $u^\infty\in H^1(B_{1})$ with Lipschitz constant $L$. 
    
    Moreover, if $u^\infty\not\equiv0$, we have $\Omega_{u^j}\to\Omega_{u^\infty}$, $\partial \Omega_{u^j}\to \partial \Omega_{u^\infty}$ locally in the Hausdorff distance in $B_1$, and $\chi_{\Omega_{u^j}}\to \chi_{\Omega_{u^\infty}}$ in $L^1_{\text{loc}}(B_1).$
\end{proposition}
\begin{proof}
    The result follows by the argument in the proof of \cite[Theorem 9.3]{kriventsov-weiss} (see also \cite[Remark 9.4]{kriventsov-weiss}).
\end{proof}
\begin{remark}\label{rem:compactness} 
We remark the differences between \cref{lemma:compactness-kw} (or \cite[Theorem 9.3]{kriventsov-weiss}) and other compactness results for the one-phase problem, such as \cite[Proposition 4.1]{jerisonkamburov}, \cite[Proposition A.4]{KamburovWang} and \cite[Lemma 4.5]{ChanFernandezFigalliSerra}. 

First, variational solutions do not satisfy in general either the non-degeneracy property, or a lower-density assumption on the zero phase. 
Therefore, the convergence of the free boundaries cannot be obtained through the standard arguments. Nevertheless, this convergence follows from the following lower-density estimate for the positivity set (see \cite[Lemma 9.2]{kriventsov-weiss}) $$|\{u>0\}\cap B_r(x_0)|\ge cr^d\quad\text{for every }x_0\in\partial \Omega_u\cap B_1,$$ 
for $r$ small enough and $c=c(d,L)>0$. This last estimate is a straightforward consequence of the Weiss' monotonicity formula, once the lower bound $\Theta(u,x_0)\ge \omega_d/2$ is known.

Second, in a compactness argument based on \cref{lemma:compactness-kw}, one needs to rule out the degenerate case $u^\infty\equiv0$. 
In the compactness argument used to prove \cref{lemma:smooth-parametrization-lemma} below (see also \cref{lemma:u-zeta-sigma}),
this is guaranteed by the convergence of the sequence to the nontrivial cone $b$ in $L^2(\partial B_1)$. Indeed, the limit satisfies $u^\infty\equiv b$ on $\partial B_1$, and thus $u^\infty\not\equiv0$.
 \end{remark}
\subsection{Existence of blow-ups}
In the following lemma, we prove the existence of the blow-up limits and the decomposition of the free boundary into regular and singular points, for variational solutions of the one-phase problem.
\begin{lemma}\label{lemma:existence-blow-up}
    Let $u\in H^1(B_1)$ be a variational solution of the one-phase problem, and let $x_0\in\partial\Omega_u\cap B_1$. Then, along a subsequence $r_j\to0^+$, the rescaling $u_{r_j,x_0}$ converges locally uniformly to a function $u_{0,x_0}\in H^1_{\text{loc}}(\R^d)$, which is a $1$-homogeneous global variational solution. Moreover, we can split the free boundary $\partial\Omega_u$ as follows $$\partial\Omega_u\cap B_1=\text{Reg}(u)\cup \text{Sing}(u),$$ where $\text{Reg}(u)$ is an open set in $\partial\Omega_u$, and it is locally given by the graph of a smooth function, while $\text{Sing}(u)$ is a closed set with locally finite $\HH^{d-1}$-measure.
\end{lemma}
\begin{proof}
    As a consequence of \cref{lemma:compactness-kw}, we have, along a subsequence $r_j\to0^+$, the locally uniform convergence of $u_{r_j,x_0}$ to some global variational solution $u_{0,x_0}$. Moreover, by the monotonicity formula in \cref{prop:monotonicity-formula}, we also have that $u_{0,x_0}$ is $1$-homogeneous. 
    
    Next we define $\text{Reg}(u)$ as the set of free boundary points $x_0\in\partial\Omega_u\cap B_1$ such that one blow-up is of the form $(x_d)_+$, up to a rotation, and we set $\text{Sing}(u):=(\partial\Omega_u\cap B_1)\setminus \text{Reg}(u)$. 
    By \cref{prop:kriventsov-weiss} (i), $\Theta(u,x_0)\in[\omega_d/2,\omega_d]$ and $x_0\mapsto\Theta(u,x_0)$ is upper semicontinuous. Then, $\partial\Omega_u\setminus\Sigma^{H}_u$ is an open set on the free boundary $\partial\Omega_u$. 
    
    Now let $x_0\in\text{Reg}(u)$, and choose $r_j\to0^+$ such that $u_{r_j,x_0}\to (x_d)_+$. In particular, $\Theta(u,x_0)=\omega_d/2$ and thus $x_0\not\in \Sigma^H_u$. Then, there exists $\rho>0$ such that $\Sigma_u^H\cap B_\rho(x_0)=\emptyset$. Therefore, for every $j$ large enough, $\Sigma_{u_{r_j,x_0}}^H\cap B_1=\emptyset$ and thus, by \cref{prop:kriventsov-weiss} (iii), $u_{r_j,x_0}$ is a viscosity solution in $B_1$. 
    
    Since $u_{r_j,x_0}\to (x_d)_+$ locally uniformly and the corresponding positivity sets and free boundaries converge locally in the Hausdorff distance (see \cref{lemma:compactness-kw}), for $j$ sufficiently large, $u_{r_j,x_0}$ satisfies the flatness assumptions of the improvement of flatness theorem in \cite{DeSilva:FreeBdRegularityOnePhase}.
    Then, we obtain that $\text{Reg}(u)$ is an open set in $\partial\Omega_u$ and it is locally $C^{1,\alpha}$. The smoothness then follows by applying the hodograph map argument in \cite[Theorem 2]{KinderlehrerNirenberg1977:AnalyticFreeBd}. Finally, $\text{Sing}(u)$ is a closed set with locally finite $\HH^{d-1}$-measure, by \cref{prop:kriventsov-weiss} (ii).
\end{proof}
\subsection{One-phase cones with isolated singularity}
We study blow-ups with isolated singularity, according to the following definition.
\begin{definition}[Cones with isolated singularity]\label{def:cone-isol-sing}
    We say that $b\in H^1_{\text{loc}}(\R^d)$ is a one-phase cone with isolated singularity, if $b\not\equiv0$ is $1$-homogeneous and is a classical solution of the one-phase problem in $\R^d\setminus \{0\}$, in the sense of \cref{def:classical-solution}.
\end{definition}

As shown in the following remark, a one-phase cone with isolated singularity can be characterized as a non-trivial $1$-homogeneous global variational solution whose only possible singular point is the origin.
\begin{remark}\label{rem:outside-origin}
    We observe that $b$ is a one-phase cone with isolated singularity (see \cref{def:cone-isol-sing}), if and only if $b$ is a non-trivial $1$-homogeneous global variational solution (see \cref{def:variational-sol}), and $\text{Sing}(b)\subset\{0\},$ where $\text{Sing}(b)$ is defined in \cref{lemma:existence-blow-up}.

     Indeed, if $b$ is a $1$-homogeneous classical solution in $\R^d\setminus \{0\}$, then $b$ is a variational solution in $\R^d\setminus{\overline B_\eps}$ for every $\eps>0$. Moreover, since $b$ is a classical solution away from the origin, then $b$ is uniformly Lipschitz continuous on $\partial B_1$, and then $\|\nabla b\|_{L^\infty(B_1)}\le L$ by homogeneity. Then, the variational identity can be extended across the origin by testing with a cutoff which vanishes in $B_\eps$ and letting $\eps\to0^+$. Thus, $b$ is a global variational solution.

     Conversely, if $b$ is a $1$-homogeneous global variational solution with $\text{Sing}(b)\subset\{0\}$ then, by \cref{lemma:existence-blow-up}, every point of $\partial\Omega_b\setminus\{0\}$ is regular. Then, $\partial\Omega_b\setminus\{0\}$ is smooth and $b$ is smooth up to the free boundary outside the origin. Thus $b$ is a classical solution in $\R^d\setminus \{0\}$.
\end{remark}
\begin{remark}
As a consequence of the previous remark, if $u\in H^1(B_1)$ is a variational solution and if $b\not\equiv0$ is a blow-up of $u$ at $0\in\partial\Omega_u$, with $\text{Sing}(b)\subset\{0\}$, then $b$ is a one-phase cone with isolated singularity. Indeed, $b$ is a global variational solution and, by the monotonicity formula in \cref{prop:monotonicity-formula}, it is $1$-homogeneous. The conclusion then follows from \cref{rem:outside-origin}.
\end{remark}

\section{The functional \texorpdfstring{$\mathcal{G}$}{G}}\label{sec2.5}
\subsection{Eigenfunctions and eigenvalues on a spherical graph}\label{subsection:expansion} Let $b$ be a one-phase cone with isolated singularity. Note that $\Omega_b\cap\partial B_1$ is connected, then $b$ can be written as a multiple of the first eigenfunction to the spherical Laplacian on its positivity set. More precisely, we consider $\varphi_1^D$ the first (normalized) eigenfunction of the spherical Laplacian in $D:=\Omega_b\cap \partial B_1$, namely $$-\Delta_\theta \varphi_1^D=(d-1)\varphi_1^D\quad\text{in }D, \qquad \varphi_1^D=0\quad\text{on }\partial D,\qquad\|\varphi_1^D\|_{L^2(\partial B_1)}=1.$$ 
Note that the normal derivative of $\varphi_1^D$ is constant on $\partial D$, then we can choose $\kappa_0>0$ such that $b=\kappa_0\varphi_1^D$, so that $$\kappa_0\partial_{\nu}\varphi_1^D=\partial_{\nu} b=-1\quad\text{on } \partial D,$$ where $\nu$ is the outward unit normal to $\partial D$.

Consider a $C^{2,\alpha}$ graph $\zeta:\partial D\to\R$, sufficiently close to $0$ in $C^{2,\alpha}$. We denote by $D_\zeta\subset\partial B_1$ the perturbed domain of $D$ whose boundary is given by \be\label{eq:graph}\partial D_\zeta=\text{graph}_{\partial D}(\zeta):=\{\exp_x(\zeta(x)\nu(x)):\ x\in\partial D\}\subset\partial B_1,\ee where 
$\exp:T\mathbb{S}^{d-1}\to\mathbb{S}^{d-1}$ is the exponential map $\exp_x(v):=\cos(|v|)x+\sin(|v|)\frac{v}{|v|}.$
In particular, we have $D_0=D$. 

We can find an $L^2(D_\zeta)$-orthonormal basis of $H^1_0(D_\zeta)$ of eigenfunctions of the spherical Laplacian $\varphi_1^{D_\zeta},\varphi_2^{D_\zeta},\ldots,\varphi_j^{D_\zeta},\ldots$ corresponding to the eigenvalues $\lambda_1(\zeta)<\lambda_2(\zeta)\le\ldots\le \lambda_j(\zeta)\le\ldots,$ namely $$-\Delta_\theta \varphi_j^{D_\zeta}=\lambda_j(\zeta)\varphi_j^{D_\zeta}\quad\text{in }D_\zeta, \qquad\varphi_j^{D_\zeta}=0\quad\text{on }\partial D_\zeta,\qquad \|\varphi_j^{D_\zeta}\|_{L^2(\partial B_1)}=1,$$ for every $j\in\N$. In particular, if $\zeta$ is sufficiently close to $0$ in $C^{2,\alpha}$ norm, then $\lambda_1(\zeta)$ is close to $\lambda_1(D)=d-1$. Then we deduce the spectral gap \be\label{eq:gap-eigenvalue}\lambda_j(\zeta)-(d-1)\ge \widetilde c>0\quad\text{for every }j\ge2,\ee for some $\widetilde c>0$ depending only on $d$ and $b$.

\subsection{The functional $\mathcal{G}$}\label{subsec:notationsX}
For small $\zeta\in C^{2,\alpha}(\partial D)$ and $\sigma\in\R$, we define the functional \be\label{def:G}\mathcal{G}(\zeta,\sigma):=(\kappa_0^2+\sigma)(\lambda_1(\zeta)-(d-1))+m(\zeta)-m(0),\ee where $m(\zeta):=\mathcal{H}^{d-1}(D_\zeta)$. Note the difference of the functional $\mathcal{G}$ with the corresponding functional in \cite{esv}, where the dependence on $\sigma$ is cubic. We refer to \cref{rem:differences-integrability} for the reason behind this difference.

We denote by $\delta\mathcal{G}(\zeta,\sigma)=(\delta_\zeta\mathcal{G}(\zeta,\sigma),\delta_\sigma\mathcal{G}(\zeta,\sigma))$ the first variation of $\mathcal{G}$, where
\bea \delta_\zeta\mathcal{G}(\zeta,\sigma)[g]:=\frac{d}{dt}\Big|_{t=0}\mathcal{G}(\zeta+tg,\sigma)\quad\text{and}\quad
     \delta_\sigma \mathcal{G}(\zeta,\sigma)[\tau]:=\frac{d}{dt}\Big|_{t=0}\mathcal{G}(\zeta,\sigma+t\tau)
     \eea for $g\in C^{2,\alpha}(\partial D)$ and $\tau\in\R$.
     We also denote by $\nabla \mathcal{G}(\zeta,\sigma)=(\nabla_\zeta \mathcal{G}(\zeta,\sigma),\partial_\sigma \mathcal{G}(\zeta,\sigma))\in L^2(\partial D)\oplus \R$ the gradient of $\mathcal{G}$, defined as $$\langle\nabla_\zeta\mathcal{G}(\zeta,\sigma),g\rangle_{L^2(\partial D)}=\delta_{\zeta}\mathcal{G}(\zeta,\sigma)[g]\quad\text{and}\quad \partial_\sigma\mathcal{G}(\zeta,\sigma)\tau=\delta_{\sigma}\mathcal{G}(\zeta,\sigma)[\tau].$$ 
 In order to explicitly compute the gradient of $\mathcal{G}$, we need to introduce some preliminary definitions from \cite[Appendix A]{esv}. For $\zeta\in C^{2,\alpha}(\partial D)$, with $\zeta$ small in $C^{2,\alpha}$, we consider the function $\Psi_{\zeta}:\partial D\to \partial D_{\zeta}$ $$\Psi_{\zeta}(x):=\exp_x(\zeta(x)\nu(x))=\cos(\zeta(x))x+\sin(\zeta(x))\nu(x),$$ where $\nu$ is the outward unit normal to $\partial D$.
 For a direction of variation $g\in C^{2,\alpha}(\partial D)$, the derivative $$\partial_t\Psi_{\zeta+tg}=g(x)Z_{\zeta+tg}(x),\quad\text{where}\quad Z_{\zeta}(x):=\cos(\zeta(x))\nu(x)-\sin(\zeta(x))x,$$ defines a vector field $X=X_{\zeta,g}$ such that $$\partial_t\Psi_{\zeta+tg}=X(\Psi_{\zeta+tg}).$$ 
 We also define $a_{\zeta}:=Z_{\zeta}\cdot \nu_\zeta(\Psi_\zeta)$, where $\nu_\zeta$ is the outward unit normal to $\partial D_\zeta$. Then \be\label{eq:quig}X(\Psi_\zeta)\cdot \nu_\zeta(\Psi_\zeta)=a_\zeta g.\ee Finally, we consider the Jacobian $J_\zeta$ defined through the change of coordinates $y=\Psi_\zeta(x)$. In particular, $d\HH^{d-2}_{\partial D_\zeta}=J_\zeta d\HH^{d-2}_{\partial D}$.

We can compute the gradient of $\mathcal{G}$ as follows. For the $\zeta$ variables, by \cite[Lemma A.1]{esv}, and if $\kappa:=(\kappa_0^2+\sigma)^{1/2}$, we have that \be\label{eq:lemmaA.1}\begin{aligned}\delta_\zeta\mathcal{G}(\zeta,\sigma)[g]&=\int_{\partial D_\zeta} (X\cdot\nu_\zeta)\,d\HH^{d-2}-\kappa^2\int_{\partial D_\zeta}\partial_X\varphi_1^{D_\zeta}\partial_{\nu_\zeta}\varphi_1^{D_\zeta}\,d\HH^{d-2}
    \\&=\int_{\partial D_\zeta} (X\cdot\nu_\zeta)\Big(1-\kappa^2(\partial_{\nu_\zeta}\varphi_1^{D_\zeta})^2\Big)\,d\HH^{d-2}.\end{aligned}\ee Then, changing variables $y=\Psi_\zeta(x)$ and using \eqref{eq:quig}, we have \be\label{eq:gradient-explicitly-zeta-before}\nabla_\zeta\mathcal{G}(\zeta,\sigma)=a_\zeta J_\zeta\Big(1-(\kappa_0^2+\sigma)(\partial_{\nu_\zeta}\varphi_1^{D_\zeta})^2\Big)\circ \Psi_\zeta,\ee 
     On the other hand, for the $\sigma$ variables, we have
    \be\label{eq:gradient-explicitly-other-before}\partial_\sigma\mathcal{G}(\zeta,\sigma)=\lambda_1(\zeta)-(d-1).\ee

\subsection{The linearized operator and its kernel}\label{def:linerized-operators}

The second variation of $\mathcal{G}$ is defined as \be\label{eq:eq45} \begin{aligned}\delta^2\mathcal{G}(\zeta,\sigma)[(g_1,\tau_1)&,(g_2,\tau_2)]:=\frac{\partial^2}{\partial t_1\partial t_2}\Big|_{t_1=t_2=0}\mathcal{G}(\zeta+t_1g_1+t_2g_2,\sigma+t_1\tau_1+t_2\tau_2)\\&=
(\kappa_0^2+\sigma)\delta^2\lambda_1(\zeta)[g_1,g_2]+\delta^2m(\zeta)[g_1,g_2]+\tau_1\delta\lambda_1(\zeta)[g_2]+\tau_2\delta\lambda_1(\zeta)[g_1].\end{aligned}\ee 
We also consider the second variation of $\mathcal{G}$ only in the $\zeta$ variables \be\label{eq:gzetazeta} \delta_{\zeta\zeta}^2\mathcal{G}(\zeta,\sigma)[g_1,g_2]:=\frac{\partial^2}{\partial t_1\partial t_2}\Big|_{t_1=t_2=0}\mathcal{G}(\zeta+t_1g_1+t_2g_2,\sigma).\ee
Throughout the paper, we will consider the following operators. First, let $\ell\in H^{-1/2}(\partial D)$ be defined as (see \cite[Equation A.10]{esv}) \be\label{eq:operator-ell}\ell[g]:=\delta\lambda_1(0)[g]=-\frac{1}{\kappa_0^2}\int_{\partial D}g\,d\HH^{d-2}.\ee
Then, we consider $L:H^{1/2}(\partial D)\to H^{-1/2}(\partial D) $ defined as (see \cite[Equation A.13]{esv})
\be\label{eq:operator-L}\begin{aligned} L[g_1,g_2]&:=\delta^2_{\zeta\zeta}\mathcal{G}(0,0)[g_1,g_2]=\kappa_0^2\delta^2\lambda_1(0)[g_1,g_2]+\delta^2m(0)[g_1,g_2]\\&=-2\int_{\partial D}g_1g_2H_{\partial D} \,d\HH^{d-2}+2\int_{\partial D}g_1Tg_2 \,d\HH^{d-2}\end{aligned}\ee 
where $H_{\partial D}\ge0$ is the scalar mean curvature of $\partial D$ oriented towards the complement of $D$, $\nu$ is the outward unit normal to $\partial D$, and $T:H^{1/2}(\partial D)\to H^{-1/2}(\partial D)$ is the Dirichlet-to-Neumann operator defined by $Tg=\partial_\nu u_g$, where $u_g$ is the solution of \be\label{eq:def-ug}\begin{cases}
        -\Delta_\theta u_g =(d-1)u_g+\ell[g]b&\quad \text{in }D,\\
        u_g=g&\quad \text{on }\partial D,\\
        \int_{ D} u_gb\,d\HH^{d-1}=0.
    \end{cases}\ee 
For the well posedness of $T$, as well as the existence and uniqueness of $u_g$, we refer to \cite[Equation A.14]{esv} and the subsequent discussion. Notice that $u_g$ is the shape derivative of $\kappa_0\varphi_1^{D_{tg}}$ in the direction $g$.

\medskip

We define the linearized operator $\mathcal{L}:H^{1/2}(\partial D)\oplus\R\to H^{-1/2}(\partial D)\oplus \R$ as \bea\label{eq:def-mathcalL}\mathcal{L}[(g_1,\tau_1),(g_2,\tau_2)]:=\delta^2\mathcal{G}(0,0)[(g_1,\tau_1),(g_2,\tau_2)]=L[g_1,g_2]+\tau_1\ell[g_2]+\tau_2\ell[g_1].\eea In particular, we can write $\mathcal{L}(g,\tau)=(Lg+\tau\ell,\ell[g])\in H^{-1/2}(\partial D)\oplus\R$.

We also define the kernel of $\mathcal{L}$ as follows \be\label{eq:kernel}K:=\ker \mathcal{L}=\{(g,\tau)\in H^{1/2}(\partial D)\oplus \R: \ Lg+\tau\ell=0,\ \ell[g]=0\}.\ee

\subsection{A useful lemma}
Let $\mathcal{L}$ be the linearized operator with kernel $K:=\ker\mathcal{L}$, defined in \cref{def:linerized-operators}.
For simplicity, we denote by $X:=H^{1/2}(\partial D)\oplus \R$, and we write $\xi=(g,\tau)\in X$. We consider the norm on $X$ $$\|\xi\|^2_X:=\|g\|_{H^{1/2}(\partial D)}^2+|\tau|^2.$$ Similarly, we denote by $\|\xi\|^2_{L^2(\partial D)\oplus\R}:=\|g\|_{L^2(\partial D)}^2+|\tau|^2$.
The dual space is $X':=H^{-1/2}(\partial D)\oplus \R$ with norm $\|\cdot \|_{X'}$.

\smallskip

In the next lemma, we summarize some results that will be useful in the following. The proof essentially follows from the results in \cite{esv}.
\begin{lemma}\label{lemma2.3}
   There exist $\overline\delta>0$ and $C>0$ depending only on $d$ and $b$ such that, for $\zeta\in C^{2,\alpha}(\partial D)$ and $\sigma\in\R$, with $\|\zeta\|_{C^{2,\alpha}(\partial D)}+|\sigma|\le \overline\delta$, the following holds.
    \begin{itemize}
        \item[(i)] The first variation of $\mathcal{G}$ vanishes at $(0,0)\in C^{2,\alpha}(\partial D)\oplus \R$, namely $\delta\mathcal{G}(0,0)=0.$
        \item[(ii)] $\delta^2\mathcal{G}(\zeta,\sigma)$ extends to a continuous symmetric bilinear form on $X\times X$. Moreover, there exists a modulus of continuity $\omega:[0,1]\to\R_+$, with $\lim_{t\to0}\omega(t)=0$, such that, for every $\xi_1,\xi_2\in X$, we have \be\label{eq:point2}\left|\delta^2\mathcal{G}(\zeta,\sigma)[\xi_1,\xi_2]-\delta^2\mathcal{G}(0,0)[\xi_1,\xi_2]\right|\le \omega\Big(\|\zeta\|_{C^{2,\alpha}(\partial D)}+|\sigma|\Big)\|\xi_1\|_{X}\|\xi_2\|_{X}.\ee 
        \item[(iii)] There exists a basis of eigenfunctions $\{\xi_i\}\subset X$, orthonormal in $L^2(\partial D)\oplus \R$, and a non-decreasing sequence of eigenvalues $\{\Lambda_i\}$, with $\Lambda_i\to+\infty$, satisfying $$\delta^2\mathcal{G}(0,0)[\xi_i,\xi_j]=\Lambda_i\delta_{ij},$$ where $\xi_i\in C^{2,\alpha}(\partial D)\oplus\R$ for every $i\in\N$. Moreover, $N:=\text{dim}\, K<+\infty$.
        \item[(iv)] 
        Let $K^\perp$ be the orthogonal of $K$ in $L^2(\partial D)\oplus\R$. Denoting by $\mathcal{L}^\perp:=\mathcal{L}|_{K^\perp\cap X}$, we have $$\|\xi^\perp\|_{X}\le C\|\mathcal{L}^\perp \xi^\perp\|_{X'}\quad\text{for every }\xi^\perp\in K^\perp\cap X.$$
    \end{itemize}
\end{lemma}
\begin{proof} First, (i) follows from the fact that $\delta_\zeta\mathcal{G}(0,0)=0$ by \cite[Lemma 2.3]{esv} and that $\delta_\sigma \mathcal{G}(0,0)=\lambda_1(0)-(d-1)=0$.

In order to prove (ii), we first show that $\delta^2\mathcal{G}(\zeta,\sigma)$ extends to a symmetric bilinear form on $X\times X$. Indeed, using the identity \eqref{eq:eq45}, it is sufficient to combine that $\delta^2\lambda_1(\zeta)$ and $\delta^2m(\zeta)$ extend to a symmetric bilinear form on $H^{1/2}(\partial D)\times H^{1/2}(\partial D)$ (see \cite[Subsection A.1]{esv}), with the bound $|\delta\lambda_1(\zeta)[g]|\le C\|g\|_{H^{1/2}(\partial D)},$ which is exactly \cite[Equation A.16]{esv}.
It remains to prove the modulus of continuity estimate in \eqref{eq:point2}. By the results in \cite[Subsection A.5]{esv}, we have the two estimates \bea|\delta^2\lambda_1(\zeta)[g_1,g_2]-\delta^2\lambda_1(0)[g_1,g_2]|\le \omega(\|\zeta\|_{C^{2,\alpha}(\partial D)})\|g_1\|_{H^{1/2}(\partial D)}\|g_2\|_{H^{1/2}(\partial D)},\eea \bea |\delta^2m(\zeta)[g_1,g_2]-\delta^2m(0)[g_1,g_2]|\le \omega(\|\zeta\|_{C^{2,\alpha}(\partial D)})\|g_1\|_{H^{1/2}(\partial D)}\|g_2\|_{H^{1/2}(\partial D)}.\eea
On the other hand, using \cite[Equation A.9]{esv}, we have for some explicit\footnote{Using the notation in \cref{subsec:notationsX}, we have $A_\zeta=-a_\zeta J_\zeta\big((\partial_{\nu_\zeta}\varphi_1^{D_\zeta})^2\circ \Psi_\zeta\big).$ Then, the estimate in \eqref{eq:second-estimate} follows from the $C^{2,\alpha}$ dependence of the terms involved in $A_\zeta$.} $A_\zeta:\partial D\to \R$, \be\label{eq:second-estimate}\delta\lambda_1(\zeta)[g]=\int_{\partial D}A_\zeta g\,d\HH^{d-2},\quad\text{where}\quad \|A_\zeta-A_0\|_{L^2(\partial D)}\le \omega(\|\zeta\|_{C^{2,\alpha}(\partial D)}).\ee 
Therefore, $\zeta\mapsto\delta\lambda_1(\zeta)$ is continuous and satisfies $$|\delta\lambda_1(\zeta)[g]-\delta\lambda_1(0)[g]|\le \omega(\|\zeta\|_{C^{2,\alpha}(\partial D)})\|g\|_{H^{1/2}(\partial D)}.$$ Finally, we can control $\sigma \delta^2\lambda_1(\zeta)[g_1,g_2]$ using again the fact that $\delta^2\lambda_1(\zeta)$ extends to a symmetric bilinear form on $H^{1/2}(\partial D)\times H^{1/2}(\partial D)$. This concludes the proof of \eqref{eq:point2}.

Next we prove (iii). Recalling the linearized operators $L$ and $\ell$ from \cref{def:linerized-operators}, we observe that $L$ is a self-adjoint operator, bounded from below, with compact resolvent in $L^2(\partial D)$ (see \cite[Subsection A.2]{esv}). Since $\mathcal{R}(g,\tau):=(\tau\ell,\ell[g])$ is bounded, finite rank and self-adjoint in $L^2(\partial D)\oplus\R$ and $\mathcal{L}(g,\tau)=(Lg+\tau\ell,\ell[g])=(Lg,0)+\mathcal{R}(g,\tau)$, then $\mathcal{L}$ is a bounded finite rank perturbation of $L\oplus0$. Thus, $\mathcal{L}$ is a self-adjoint operator with compact resolvent in $L^2(\partial D)\oplus \R$ as well. Therefore, there exists an orthonormal basis of eigenfunctions $\{\xi_j\}$, with eigenvalues $\{\Lambda_j\}$, $\Lambda_j\to+\infty$, and moreover every eigenspace has finite dimension. In particular, $N:=\text{dim}\, K<+\infty$. Finally, the $C^{2,\alpha}$ regularity of the eigenfunctions follows from standard elliptic regularity.

In order to prove (iv), we first observe that, by the ellipticity of the Dirichlet-to-Neumann operator
$$\|\xi\|_{X}\le C\Big(\|\mathcal L\xi\|_{X'}+\|\xi\|_{L^2(\partial D)\oplus\R}\Big)\quad\text{for every }\xi\in X.$$
Since $H^{1/2}(\partial D)\hookrightarrow L^2(\partial D) $ is compact, the second term in the right-hand side above can be removed in $K^\perp$ by a standard contradiction argument. This concludes the proof.
\end{proof}

\subsection{Lyapunov-Schmidt reduction}

Let $K:=\ker\mathcal{L}$ be the kernel of the second variation, defined in \eqref{eq:kernel}. Its dimension is $N:=\text{dim}\, K<+\infty$, which is finite by \cref{lemma2.3} (iii).
We denote by $P_K$ and $P_{K^\perp}$ the $L^2(\partial D)\oplus \R$ projections on $K$ and $K^\perp$.

The following proposition is a standard application of the Lyapunov-Schmidt reduction. The proof is a slight modification of \cite[Section 2]{Simon} or \cite[Lemma B.1]{esv}.

\begin{proposition}\label{lemma:Lyapunov-Schmidt}
			There exist a neighborhood $U$ of $0$ in $C^{1,\alpha}\oplus\R$ and an analytic map $$Y:K\cap U \to K^\perp\cap( C^{2,\alpha}(\partial D)\oplus\R)\subset X$$ 
            such that the following holds.
            \begin{itemize}
            \item $Y(0)=0$, $D Y(0)=0$. Moreover,
            \bea\label{eq:cond-on-perp}
			{P}_{K^\perp}\big(\nabla\mathcal{G}(\xi+Y(\xi))\big)=0,\quad \mbox{for every }\xi \in K\cap U.
            \eea
    \item Let $\Xi_1,\dots,\Xi_N$ be an orthonormal basis of $K$. Then, there exists $\rho>0$ such that, for every $\mu=(\mu_1,\ldots,\mu_N)\in B_\rho\subset\R^N$, the reduced functional $G:B_\rho\to\R$ defined by
\bea\label{def:G-finito-dim}
G(\mu):=\mathcal{G}(\Xi_\mu +Y(\Xi_\mu)),
\quad \mbox{with}\quad
\Xi_\mu:=\sum_{j=1}^N\mu_j\Xi_{j},
\eea
satisfies ${P}_{K}\big(\nabla\mathcal{G}(\Xi_\mu+Y(\Xi_\mu))\big)=\nabla G(\mu),$ for every $\mu\in B_{\rho} $.
			\end{itemize}
            Moreover, all the critical points of $\mathcal{G}$ in a neighborhood $V$ of $0$ in $C^{2,\alpha}\oplus\R$ are given by  \be\label{eq:Lyapunov-Schmidt-identity}\{\xi+Y(\xi): \xi=\Xi_\mu\in K\cap V, \ \nabla G(\mu)=0\}.\ee
		\end{proposition}
        Unlike \cite[Lemma B.1]{esv}, where the Lyapunov-Schmidt reduction acts only on the free boundary graph, here the reduction is performed for the full pair $(\zeta,\sigma)$. 
        This reflects the fact that $\mathcal{G}$ depends linearly on $\sigma$, rather than cubically as in \cite{esv}; see also \cref{rem:differences-integrability}.

\smallskip
        
        We also point out that the analyticity of the map $Y$ in \cref{lemma:Lyapunov-Schmidt} will be crucial to apply a finite dimensional Łojasiewicz inequality in the proof of \cref{thm:loj-ineq} below, and it follows from the fact that $\mathcal{G}$ is analytic in a neighborhood of $0$.

\section{Integrability condition}\label{sec:int}
\subsection{Integrability of one-phase cones}
Let $b$ be a one-phase cone with isolated singularity. The linearized problem at $b$ is the following (see for instance \cite{CaffarelliJerisonKenig04:NoSingularCones3D,JerisonSavin15:NoSingularCones4D,fernandezyu-linearized})          \be\label{eq:linearizzato}\begin{cases}
        \Delta v=0&\quad \text{in }\Omega_b,\\
        \partial_\nu v=Hv&\quad \text{on }\partial \Omega_b.
    \end{cases}\ee
    Here $\nu$ is the outward unit normal to $\partial\Omega_b$ and $H$ is the scalar mean curvature of $\partial\Omega_b$ oriented towards the complement of $\Omega_b$. We recall that $H\ge0$, with $H>0$ if and only if $\Omega_b$ is not a half-space. Using the language of minimal surface theory, solutions of the linearized problem \eqref{eq:linearizzato} will be called \emph{Jacobi fields}.

\medskip

We use the following definition of integrability.
\begin{definition}[Integrability condition]\label{def:integrability}
    Let $b$ be a one-phase cone with isolated singularity. We say that $b$ is integrable if for every $1$-homogeneous solution $v$ of \eqref{eq:linearizzato}, there exists a family $\{u_t\}_{t\in(0,1)}$ of $1$-homogeneous classical solutions of the one-phase problem in $\R^d\setminus \{0\}$ (see \cref{def:classical-solution}), satisfying $u_0=b$ and $\partial_tu_t|_{t=0}=v$ in the following sense $$\lim_{t\to0^+}\|u_t-b\|_{L^2(\partial B_1)}=0\quad\text{and}\quad \lim_{t\to0^+}\left\|\frac{u_t-b}{t}-v\right\|_{L^2(D)}=0,$$ where $D:=\Omega_b\cap\partial B_1$.

If for every $1$-homogeneous Jacobi field $v$, the family $\{u_t\}_{t\in(0,1)}$ is given by rotations of $b$, we say that $b$ is integrable through rotations.
\end{definition} 
We next discus some different integrability conditions appearing in the literature for the one-phase problem.

\begin{remark}[Different notions of integrability]\label{rem:differences-integrability}
We emphasize the difference between the notion of integrability used in
\cref{def:integrability} and that of integrability through rotations appearing in
\cite[Definition 2.3]{ERZ25} and \cite[Definition 2.4]{esv}.
The first one coincides with our notion of integrability through rotations, and is the natural analogue of the one in minimal surface theory (see \cite{AA81,adamsimon}). The second one is expressed in terms of the kernel of $\delta^2_{\zeta\zeta}\mathcal{G}(0,0)$, defined in \eqref{eq:gzetazeta}, and is stronger than the first, as discussed in \cite[Remark 2.4]{ERZ25}.
   
Throughout the section, we relate the integrability condition in \cref{def:integrability} to the kernel of the full second variation $\delta^2\mathcal{G}(0,0)$, which is taken
with respect to the full set of variables $(\zeta,\sigma)$, and not only with respect to $\zeta$, as in \cite{esv}. This full kernel is the one naturally associated with the $1$-homogeneous Jacobi fields (see \cref{lemma:correspondence}).
This is the reason why our choice of parametrization of the functional $\mathcal G$ is different from that in \cite{esv}. Indeed, in that paper, the corresponding
functional depends on $\sigma$ through $\sigma^3$, whereas our functional depends
linearly on $\sigma$.   

Moreover, the choice to parametrize $\mathcal{G}$ as in \eqref{def:G} is also natural since critical points $(\zeta,\sigma)$ of $\mathcal G$ correspond to $1$-homogeneous solutions $u^{\zeta,\sigma}$ of the one-phase problem (see \cref{lemma:u-zeta-sigma}). More precisely, by \eqref{eq:gradient-explicitly-zeta-before} and \eqref{eq:gradient-explicitly-other-before}, the condition $\delta_\sigma\mathcal{G}(\zeta,\sigma)=0$ implies that $u^{\zeta,\sigma}$ is harmonic in its positivity set, while $\delta_\zeta\mathcal{G}(\zeta,\sigma)=0$ implies the free boundary condition $|\nabla u^{\zeta,\sigma}|=1$ on $\partial\Omega_{u^{\zeta,\sigma}}$. If the dependence of $\mathcal{G}$ were cubic, the first condition would be lost.
\end{remark}

\subsection{A characterization of the integrability condition}\label{subsection:correspondence}
In this subsection, we characterize the integrability condition in \cref{def:integrability} showing that it is equivalent to the vanishing of the reduced functional $G$ in the Lyapunov-Schmidt reduction \cref{lemma:Lyapunov-Schmidt}.
This is the analogue of \cite[Lemma 1]{adamsimon} in the minimal surface theory.

\begin{proposition}\label{lemma:adam-simon}
A one-phase cone $b$ with isolated singularity is integrable if and only if the reduced functional $G$ in \cref{lemma:Lyapunov-Schmidt} satisfies $G\equiv0$ in a neighborhood of $0$.
\end{proposition}
In order to prove \cref{lemma:adam-simon}, in the following lemma we show a one-to-one correspondence between the kernel $K:=\ker \mathcal{L}$ of the second variation of $\mathcal{G}$ (see \cref{def:linerized-operators}) and the 1-homogeneous Jacobi fields. 
\begin{lemma}\label{lemma:correspondence}
    The elements $(g,\tau)\in K:=\ker \mathcal{L}$ are in one-to-one correspondence with $1$-homogeneous Jacobi fields, namely $1$-homogeneous solutions of \eqref{eq:linearizzato}. This correspondence is given by $K\ni (g,\tau)\mapsto v_{g,\tau}:=u_g+\frac{\tau}{2\kappa_0^2}b$, where $u_g$ is defined in \eqref{eq:def-ug}.
\end{lemma}
\begin{proof}
    Let $(g,\tau)\in K$, so that $Lg+\tau\ell=0$ and $\ell[g]=0$. We consider the solution $u_g$ of \eqref{eq:def-ug}, and we set $v_{g,\tau}:=u_g+\frac{\tau}{2\kappa_0^2}b$. We claim that the $1$-homogeneous extension of $v_{g,\tau}$ solves \eqref{eq:linearizzato}.
    
    Since $\ell[g]=0$, then    
    $$\begin{cases}
        -\Delta_\theta v_{g,\tau} =(d-1)v_{g,\tau}&\quad \text{in }D,\\
        v_{g,\tau}=g&\quad \text{on }\partial D.
    \end{cases}$$
    Let $T:H^{1/2}(\partial D)\to H^{-1/2}(\partial D)$ be the Dirichlet-to-Neumann operator defined by $Tg=\partial_\nu u_g$. Then, by \eqref{eq:operator-ell} and \eqref{eq:operator-L}, for every $\zeta\in H^{1/2}(\partial D)$, we have \be\label{eq:sopra}\begin{aligned}0=(Lg+\tau\ell)[\zeta]&=-2\int_{\partial D}\zeta gH_{\partial D}\,d\HH^{d-2}+2\int_{\partial D}\zeta T g\,d\HH^{d-2}-\frac{\tau}{\kappa_0^2} \int_{\partial D}\zeta\,d\HH^{d-2}\\&=2\int_{\partial D}\zeta\left(-u_gH_{\partial D}+\partial_\nu u_g-\frac{\tau}{2\kappa_0^2}\right)\,d\HH^{d-2}\\&=2\int_{\partial D}\zeta\left(-v_{g,\tau}H_{\partial D}+\partial_\nu v_{g,\tau}\right)\,d\HH^{d-2},\end{aligned}\ee where in the last equality we used that $b=0$ and $\partial_\nu b=-1$ on $\partial D$. By the arbitrariness of $\zeta$, we obtain that $\partial_{\nu}v_{g,\tau}=H_{\partial D}v_{g,\tau}$ on $\partial D$.
    This concludes the proof of the claim.

    Conversely, consider a 1-homogeneous solution $v$ of \eqref{eq:linearizzato}. 
    The trace of $v$ (still denoted by $v$) solves $$\begin{cases}
        -\Delta_\theta v=(d-1)v&\quad \text{in }D,\\
        \partial_\nu v=H_{\partial D}v&\quad \text{on }\partial D.
    \end{cases}$$
     We set $g:=v|_{\partial D}$. A simple integration by parts gives that $$\int_{\partial D}\Big(b\partial_\nu v-v\partial_\nu b\Big)\,d\HH^{d-2}=\int_{D}\Big(b\Delta_\theta v-v\Delta_\theta b\Big)\,d\HH^{d-1}=0.$$ Using also that $b=0$ and $\partial_\nu b=-1$ on $\partial D$, we deduce that $$0=\int_{\partial D}v\,d\HH^{d-2}=\int_{\partial D}g\,d\HH^{d-2}=-\kappa_0^2\ell[g],$$ where in the last equality we used \eqref{eq:operator-ell}. Then $\ell[g]=0$, and thus $-\Delta_\theta u_g=(d-1)u_g$ in $D$.
    Therefore, $-\Delta_\theta (v-u_g)=(d-1)(v-u_g)$ in $D$ and $v-u_g=0$ on $\partial D$. Then $v-u_g=cb$, for some $c\in\R$, and we set $\tau:=2\kappa_0^2c$. Then, since $\ell[g]=0$ and, arguing as in \eqref{eq:sopra}, $Lg+\tau\ell=0$, it follows that $(g,\tau)\in K$. This concludes the proof. 
\end{proof}
In the next lemma we prove that $1$-homogeneous solutions of the one-phase problem correspond to the vanishing of the first variation of $\mathcal{G}$.
\begin{lemma}\label{lemma:u-zeta-sigma}
    Critical points $(\zeta,\sigma)\in C^{2,\alpha}(\partial D)\oplus\R$ of $\mathcal{G}$ correspond to $1$-homogeneous solutions of the one-phase problem in the following sense.
    \begin{itemize}
        \item Let $(\zeta,\sigma)\in C^{2,\alpha}(\partial D)\oplus \R$ be sufficiently small, with $\delta\mathcal{G}(\zeta,\sigma)=0$, then $$u^{\zeta,\sigma}(r,\theta):=r(\kappa_0^2+\sigma)^{1/2}\varphi_1^{D_{\zeta}}(\theta)$$ is a $1$-homogeneous classical solution in $\R^d\setminus\{0\}$, in the sense of \cref{def:classical-solution}.
        \item Conversely, if $u$ is a $1$-homogeneous classical solution in $\R^d\setminus\{0\}$, and $u$ is sufficiently close to $b$ in $L^2(\partial B_1)$, then there exists $(\zeta,\sigma)\in C^{2,\alpha}(\partial D)\oplus\R$ sufficiently small such that $u=u^{\zeta,\sigma}$ and $\delta\mathcal{G}(\zeta,\sigma)=0$. 
    \end{itemize}
\end{lemma}
\begin{proof}
    We first assume that $\delta\mathcal{G}(\zeta,\sigma)=0$. Then, $u:=u^{\zeta,\sigma}$ is harmonic in its positivity set $\Omega_u$. Indeed, since $0=\partial_\sigma \mathcal{G}(\zeta,\sigma)=\lambda_1(\zeta)-(d-1)$, then $$-\Delta_\theta \varphi_1^{D_{\zeta}}=\lambda_1(\zeta)\varphi_1^{D_{\zeta}}=(d-1)\varphi_1^{D_{\zeta}}\quad\text{in }D_\zeta.$$ 
    Moreover, by \eqref{eq:lemmaA.1}, we have for every $g\in H^{1/2}(\partial D)$ $$0=\delta_\zeta\mathcal{G}(\zeta,\sigma)[g]=\int_{\partial D_\zeta}g_\zeta \left(1-(\kappa_0^2+\sigma)(\partial_{\nu_\zeta}\varphi_1^{D_\zeta})^2\right)\,d\HH^{d-2}$$
    where $g_\zeta=X_{\zeta,g}\cdot\nu_{\zeta}$. By the arbitrariness of $g_\zeta$, we obtain that $(\kappa_0^2+\sigma)|\nabla_{\theta} \varphi_1^{D_\zeta}|^2=1$ on $\partial D_\zeta.$ This implies that $|\nabla u|=1$ on $\partial \Omega_u\setminus\{0\}$, and thus $u$ is a $1$-homogeneous classical solution of the one-phase problem in $\R^d\setminus\{0\}$. 

We now prove the converse implication. We first claim that $\partial\Omega_u\cap \partial B_1$ is the graph over $\partial D$ of a $C^{2,\alpha}$ function $\zeta$, with small $C^{2,\alpha}$ norm.

Suppose by contradiction that there exists a sequence $\{u^j\}$ of $1$-homogeneous classical solutions in $\R^d\setminus\{0\}$, such that $u^j\to b $ in $L^2(\partial B_1)$ as $j\to+\infty$, and $\partial\Omega_{u^j}\cap\partial B_1$ is not a small graph on $\partial D$.
Since the functions $u^j$ are classical solutions in $A:=B_2\setminus \overline B_{1/2}$, they are uniformly Lipschitz there. Then, by \cref{lemma:compactness-kw} and using that $u^j$ and $b$ are $1$-homogeneous, we deduce that $u^j\to b$ uniformly in $A$, 
and moreover $\Omega_{u^j}\to\Omega_{b}$, $\partial \Omega_{u^j}\to \partial \Omega_{b}$ locally in the Hausdorff distance in $A$, and $\chi_{\Omega_{u^j}}\to \chi_{\Omega_{b}}$ in $L^1_{\text{loc}}(A).$ Therefore, since $b$ is regular in $A$, we can apply the improvement of flatness theorem in \cite{DeSilva:FreeBdRegularityOnePhase} and then the hodograph map argument in \cite{KinderlehrerNirenberg1977:AnalyticFreeBd} to deduce that $\partial\Omega_{u^j}\cap\partial B_1$ is given by a $C^{2,\alpha}$ graph, with small $C^{2,\alpha}$ norm. This gives the desired contradiction.

Then, the spherical free boundary of $u$ is
a normal graph over $\partial D$, with small
$C^{2,\alpha}$ norm. Moreover, since $u$ is a $1$-homogeneous harmonic function in $\Omega_u$, then $u|_{\partial B_1}$ is a multiple of $\varphi_1^{D_\zeta}$. Then, since $u$ is close to $b$ in $L^2(\partial B_1)$, we can find $\sigma\in\R$ small such that $u=u^{\zeta,\sigma}$. Finally, arguing as in the first part of the proof, one can prove that $\delta\mathcal{G}(\zeta,\sigma)=0$.
\end{proof}
\begin{remark}\label{finalremark}
We notice that, for every $\xi=(g,\tau)\in K$ and $(\zeta,\sigma):=\xi+Y(\xi)$, we have the expansion $$u^{\zeta,\sigma}=b+v_{g,\tau}+o(|(g,\tau)|),\quad\text{as }(g,\tau)\to 0\text{ in }K,$$
where $u^{\zeta,\sigma}$ and $v_{g,\tau}$ are defined in \cref{lemma:u-zeta-sigma} and \cref{lemma:correspondence} respectively.
Indeed, by using that $u_g$ defined in \eqref{eq:def-ug} is the shape derivative of $\kappa_0\varphi_1^{D_{tg}}$ in the direction $g$, we obtain $$u^{tg,t\tau}=b+tv_{g,\tau}+o(t)\quad\text{as }t\to0^+,$$ by definition of $v_{g,\tau}$. Moreover, by point (iii) in \cite[Lemma 2.3]{esv}, for every $(\zeta_1,\sigma_1)$, $(\zeta_2,\sigma_2)\in C^{2,\alpha}(\partial D)\oplus\R$ sufficiently close to zero, we have the Lipschitz estimate $$\|u^{\zeta_1,\sigma_1}-u^{\zeta_2,\sigma_2}\|_{L^2(\partial B_1)}\le C\Big(\|\zeta_1-\zeta_2\|_{L^2(\partial D)}+|\sigma_1-\sigma_2|\Big).$$ 
Since $K$ is finite dimensional, the Lipschitz estimate makes the directional expansion uniform over the unit sphere of $K$. Hence
    $$u^{g,\tau}=b+v_{g,\tau}+o(|(g,\tau)|),\quad\text{as }(g,\tau)\to 0\text{ in }K.$$ Finally, using again the Lipschitz estimate above and that $DY(0)=0$, we conclude.
\end{remark}
Now we prove the characterization of the integrability condition in \cref{lemma:adam-simon}.

\begin{proof}[Proof of \cref{lemma:adam-simon}]
    We divide the proof into two steps.

    \smallskip
\noindent\textit{Step 1.}
    We first suppose that $G\equiv0$ in a neighborhood of $0$.    
    Take a $1$-homogeneous Jacobi field $v$, namely a $1$-homogeneous solution of \eqref{eq:linearizzato}. By \cref{lemma:correspondence}, there exists a corresponding $\xi=(g,\tau)\in K$. Moreover, since $\nabla G\equiv0$ in a neighborhood of $0$,  by \cref{lemma:Lyapunov-Schmidt} we have $\delta\mathcal{G}(t\xi+Y(t\xi))=0,$ for every $t>0$ sufficiently small. 
    We consider $$u_t(r,\theta)=r(\kappa_0^2+\sigma(t))^{1/2}\varphi_1^{D_{\zeta(t)}}(\theta),\quad\text{where} \quad (\zeta(t),\sigma(t))=t\xi+Y(t\xi).$$
    We claim that $u_t$ is a family of solutions generating $v$, proving that $b$ is integrable, by the arbitrariness of $v$.
    
    Since $\delta\mathcal{G}(\zeta(t),\sigma(t))=\delta\mathcal{G}(t\xi+Y(t\xi))=0$, by \cref{lemma:u-zeta-sigma}, $u_t$ is a $1$-homogeneous classical solution of the one-phase problem in $\R^d\setminus\{0\}$. We observe that $(\zeta(0),\sigma(0))=(0,0)$, then $u_0=b$. Moreover, by \cref{finalremark}, $\partial_t u_t|_{t=0}=rv_{g,\tau}(\theta)$. This proves that the family $\{u_t\}$ generates $v$, as claimed.

\smallskip
\noindent\textit{Step 2.}
    Conversely, assume that $b$ is integrable. We suppose by contradiction that $G\not\equiv0$ in a neighborhood of $0$. Then, we can write \bea G(\mu)=\sum_{j\ge p}G_j(\mu),\eea
    where $G_j$ are $j$-homogeneous and $G_p\not\equiv0$, for some $p\ge3$. In particular, there exists $\mu\in\R^N$ such that $\nabla G_p(\mu)\not=0$. 
    
    We take the corresponding $\xi=\sum_{j=1}^N\mu_j\Xi_j$, where $\Xi_1,\ldots,\Xi_N$ is an orthonormal basis of $K$.
    Then, we consider the $1$-homogeneous Jacobi field $v$ given by \cref{lemma:correspondence} corresponding to $\xi=(g,\tau)$. By the integrability condition, there exists a family of $1$-homogeneous solutions $\{u_t\}$ such that $u_0=b$ and $\partial_t u_t|_{t=0}=v_{g,\tau}$. 
    
    By \cref{lemma:u-zeta-sigma}, if $t>0$ is sufficiently small, we can find $(\zeta(t),\sigma(t))\in C^{2,\alpha}(\partial D)\oplus \R$, satisfying $\delta\mathcal{G}(\zeta(t),\sigma(t))=0$, such that $u_t=u^{\zeta(t),\sigma(t)}$. Moreover, by \eqref{eq:Lyapunov-Schmidt-identity}, we can write $(\zeta(t),\sigma(t))=\xi(t)+Y(\xi(t))$, where $\xi(t)=(g(t),\tau(t))$.
    Next, we claim that $\zeta(t)$ and $\sigma(t)$ are differentiable at $t=0$, with $\zeta'(0)=g$ and $\sigma'(0)=\tau$. In order to prove the claim, by the integrability condition and \cref{finalremark}, we have $$v_{g(t),\tau(t)}+o(|(g(t),\tau(t))|)=u_t-b=tv_{g,\tau}+o(t)\quad\text{in }L^2(D).$$ Since the map $K\ni(g,\tau)\mapsto v_{g,\tau}$ in \cref{lemma:correspondence} is injective and $K$ is a finite dimensional space, then $$|(g,\tau)|\le C\|v_{g,\tau}\|_{L^2(D)}.$$ Then, the previous expansion first gives $|(g(t),\tau(t))|=O(t)$, and then $(g(t),\tau(t))=t(g,\tau)+o(t)$. Therefore, since $D Y(0)=0$, we have $(\zeta(t),\sigma(t))=t(g,\tau)+o(t)$, concluding the proof of the claim.
    
    Finally, we observe that
    $$\xi(0)=(g(0),\tau(0))=(0,0)\quad\text{and}\quad \xi'(0)=P_K( \zeta'(0), \sigma'(0))=P_K\xi=\xi.$$
    Moreover, we can write $\xi(t)=\sum_{j=1}^N\mu_j(t)\Xi_j$ for some $\mu(t)\in\R^N$ satisfying $\mu(t)=\mu t+o(t)$ as $t\to0^+$.    
    By \eqref{eq:Lyapunov-Schmidt-identity}, we have $\nabla G(\mu(t))=0$. Dividing by $t^{p-1}$ and sending $t\to0^+$, we deduce that $\nabla G_p(\mu)=0$, which is a contradiction. 
\end{proof}

\section{\L ojasiewicz inequality}\label{sec3}
In this section we prove an infinite dimensional \L ojasiewicz inequality for the following spherical Weiss' energy, defined for $\phi\in H^1(\partial B_1)$ as \bea\label{def:F} \mathcal{F}(\phi):=\int_{\partial B_1}\Big(|\nabla_\theta \phi|^2-(d-1)\phi^2\Big)\,d\HH^{d-1}+\HH^{d-1}(\{\phi>0\}\cap\partial B_1).\eea 
This is precisely the spherical energy coming from the slicing lemma \cite[Lemma 12.10]{Velichkov:RegularityOnePhaseFreeBd}. 

We point out that the functional $\mathcal{F}$ is not differentiable. Then, in order to prove a \L ojasiewicz inequality, we introduce an auxiliary gradient map adapted to the decomposition into the first mode and the higher modes, which we call the \emph{reduced gradient} of $\mathcal{F}$ (see
\cref{subsection-important} below). 

 \medskip
 
The main result of this section is the following \L ojasiewicz inequality for $\mathcal{F}$.
\begin{proposition}[\L ojasiewicz inequality]\label{thm:loj-ineq}
    There exist constants $\delta_0>0$ and $\theta\in(0,\frac12]$ depending only on $d$ and $b$, and there exists a constant $C>0$ depending only on $d$, $b$ and $C_0$ such that the following holds. 
    
    Let $b$ be a one-phase cone with isolated singularity, and let $\phi\in H^1(\partial B_1)$ be a non-negative trace satisfying \be\label{eq:closass}\|\phi-b\|_{L^2(\partial B_1)}\le \delta_0\quad\text{and}\quad |\mathcal{F}(\phi)-\mathcal{F}(b)|\le C_0,\ee
    for some $C_0>0$. 
    We also assume that \be\label{eq:graph-assumption}\partial\Omega_\phi\cap \partial B_1=\text{graph}_{\partial D}(\zeta)=\partial D_\zeta,\ D_\zeta:=\Omega_\phi\cap \partial B_1,\text{ for some }\zeta:\partial D\to\R, \ \|\zeta\|_{C^{2,\alpha}(\partial D)}\le\delta_0,\ee where $D:=\Omega_b\cap\partial B_1$ and $\text{graph}_{\partial D}(\zeta)$ is defined in \eqref{eq:graph}. Then,
    $$|\mathcal{F}(\phi)-\mathcal{F}(b)|^{1-\theta}\le C\|\nabla_{\rm red}\mathcal{F}(\phi)\|,$$ where the reduced gradient $\nabla_{\rm red}\mathcal{F}$ and its norm are defined in \eqref{eq:gradient-reduced} and \eqref{eq:normnablaF} below.

    Moreover, if $b$ is integrable (see \cref{def:integrability}), then we can take $\theta=\frac12$.
\end{proposition}
\subsection{The reduced gradient of $\mathcal{F}$}\label{subsection-important}
Since the spherical energy $\mathcal{F}$ is not differentiable (see the discussion below), we define a reduced first variation of $\mathcal{F}$ at some point $\phi$ by rewriting the energy in terms of the graph of the free boundary and the expansion of $\phi$.
More precisely, let $\phi\in H^1(\partial B_1)$ be a non-negative trace sufficiently close to $b$ in $L^2(\partial B_1)$ and such that $\partial\Omega_\phi\cap \partial B_1$ is given by a small $C^{2,\alpha}$ graph $\zeta$ on $\partial D$.
Then, we can expand $\phi$ using the eigenfunctions of $D_\zeta$ as in \cref{subsection:expansion}, namely \bea\label{eq:expansion-phi}\phi=(\kappa_0^2+\sigma)^{1/2}\, \varphi_1^{D_\zeta}+\phi_>,\quad\text{where}\quad \phi_>=\sum_{j=2}^{\infty}c_j\varphi_j^{D_\zeta}\in E^\zeta_>:=\text{span}\{\varphi^{D_\zeta}_j\}_{j\ge2}\eea for some $\sigma\in\R$ small and $\phi_>\in E_>^\zeta$. 

With a slight abuse of notation, we identify $\phi$ with the triple $(\zeta,\sigma,\phi_>)$, keeping in mind that the high-mode space $E_>^\zeta$ depends on $\zeta$. This parametrization allows us to consider inner variations, for the $\zeta $ variables, and outer variations, at fixed $\zeta$, for the $\sigma$ and $\phi_>$ variables.

Notice that a variation in the $\phi_>$ variable could modify the positivity set, and hence it is not, in general, an admissible outer variation of the original functional $\mathcal{F}$. Moreover, the ordered higher eigenvalues $\lambda_j(\zeta)$, $j\ge2$, are not differentiable in general. This issue arises in the presence of multiple eigenvalues, since in general there is no canonical differentiable choice of the corresponding eigenfunctions. Therefore, we differentiate only the first eigenvalue $\lambda_1(\zeta)$, 
and we treat the high-mode contribution through variations in the $\phi_>$ variable, using the spectral gap in \eqref{eq:gap-eigenvalue}. 

For these two reasons, we first decompose the energy using the orthogonality of the eigenfunctions of $D_\zeta$ as follows \be\label{eq:decomposition}\mathcal{F}(\phi)=\mathcal{F}(b)+\mathcal{G}(\zeta,\sigma)+\mathcal{Q}_\zeta(\phi_>),\ee where $\mathcal{G}(\zeta,\sigma)$ is defined in \eqref{def:G}, and
$$\mathcal{Q}_\zeta(\phi_>):=\int_{D_\zeta}\Big(|\nabla_\theta\phi_>|^2-(d-1)\phi_>^2\Big)\,d\HH^{d-1}=\sum_{j=2}^\infty(\lambda_j(\zeta)-(d-1))c_j^2.$$ 
Then, we consider parameter variations $g\in C^{2,\alpha}(\partial D)$, $\tau\in\R$, $\psi_>\in E_>^\zeta$, and we define the reduced first variations of $\mathcal{F}$ 
\bea\label{eq:first-variation-zeta} \delta_{\zeta,{\rm red}}\mathcal{F}(\phi)[g]:=\delta_\zeta\mathcal{G}(\zeta,\sigma)[g]=\frac{d}{dt}\Big|_{t=0}\mathcal{G}(\zeta+tg,\sigma),\eea
\begin{equation*}\label{eq:first-variation-other2}
     \delta_{\sigma,{\rm red}} \mathcal{F}(\phi)[\tau]:= \delta_\sigma \mathcal{G}(\zeta,\sigma)[\tau]=\frac{d}{dt}\Big|_{t=0}\mathcal{G}(\zeta,\sigma+t\tau),
\end{equation*}
\bea\label{eq:first-variation-other1}
\delta_{\phi_>,{\rm red}} \mathcal{F}(\phi)[\psi_>]:=\delta_{\phi_>}\mathcal{Q}_\zeta(\phi_>)[\psi_>]=\frac{d}{dt}\Big|_{t=0}\mathcal{Q}_\zeta(\phi_>+t\psi_>).
\eea 

We denote by $\nabla_{\zeta,{\rm red}} \mathcal{F}(\phi)\in L^2(\partial D)$, $\partial_{\sigma,{\rm red}}\mathcal{F}(\phi)\in \R$ and $D_{\phi_>,{\rm red}} \mathcal{F}(\phi)\in H^{-1}(D_\zeta)$ the corresponding components of the reduced gradient of $\mathcal{F}$,
namely \bea\label{eq:gradient-partial-zeta}\langle\nabla_{\zeta,{\rm red}}\mathcal{F}(\phi),g\rangle_{L^2(\partial D)}=\delta_{\zeta,{\rm red}}\mathcal{F}(\phi)[g],\eea
\bea\label{eq:gradient-partial-other1}\partial_{\sigma,{\rm red}}\mathcal{F}(\phi)\tau=\delta_{\sigma,{\rm red}}\mathcal{F}(\phi)[\tau],\eea
\bea\label{eq:gradient-partial-other2}
 \langle D_{\phi_>,{\rm red}}\mathcal{F}(\phi),\psi_>\rangle_{H^{-1}( D_\zeta)}=\delta_{\phi_>,{\rm red}}\mathcal{F}(\phi)[\psi_>].\eea
Using the notation in \cref{subsec:notationsX}, we can explicitly compute the components of the reduced gradient of $\mathcal{F}$. Precisely, by \eqref{eq:gradient-explicitly-zeta-before} and \eqref{eq:gradient-explicitly-other-before}, we have
     \be\label{eq:gradient-explicitly-zeta}\nabla_{\zeta,{\rm red}} \mathcal{F}(\phi)=a_\zeta J_\zeta\Big(\Big(1-(\kappa_0^2+\sigma)(\partial_{\nu_\zeta}\varphi_1^{D_\zeta})^2\Big)\circ \Psi_\zeta\Big),\ee \be\label{eq:gradient-explicitly-other}\partial_{\sigma,{\rm red}}\mathcal{F}(\phi)=\lambda_1(\zeta)-(d-1)\quad\text{and}\quad D_{\phi_>,{\rm red}} \mathcal{F}(\phi)=2(-\Delta_\theta -(d-1))\phi_>.\ee 
We define the reduced gradient of $\mathcal{F}$ as \be\label{eq:gradient-reduced}\nabla_{\rm red}\mathcal{F}(\phi)=(\nabla_{\zeta,{\rm red}}\mathcal{F}(\phi),\, \partial_{\sigma,{\rm red}}\mathcal{F}(\phi),\, D_{\phi_>,{\rm red}}\mathcal{F}(\phi)),\ee and we consider its norm \be\label{eq:normnablaF}\|\nabla_{\rm red} \mathcal{F}(\phi)\|^2:=\|\nabla_{\zeta,{\rm red}}\mathcal{F}(\phi)\|_{L^2(\partial D)}^2+|\partial_{\sigma,{\rm red}}\mathcal{F}(\phi)|^2+\|D_{\phi_>,{\rm red}}\mathcal{F}(\phi)\|_{H^{-1}(D_\zeta)}^2.\ee

\subsection{The higher modes} In the following lemma, we show an estimate for the higher modes. 
\begin{lemma}\label{lemma:lemma1-loj}
    Under the assumptions of \cref{thm:loj-ineq}, there exists a constant $C>0$ depending only on $d$ and $b$ such that \bea\label{eq:highermodes}|\mathcal{Q}_\zeta(\phi_>)|^{1/2}\le C\|D_{\phi_>,{\rm red}} \mathcal{F}(\phi)\|_{H^{-1}(D_\zeta)}.\eea
\end{lemma}
\begin{proof}
First, we observe that, by \eqref{eq:gradient-explicitly-other} $$D_{\phi_>,{\rm red}} \mathcal{F}(\phi)=2(-\Delta_\theta-(d-1))\phi_>=2\sum_{j=2}^\infty(\lambda_j(\zeta)-(d-1))c_j\varphi_j^{D_{\zeta}}.$$ 
Therefore, we have \be\label{eq:Qzeta-eq0}\|D_{\phi_>,{\rm red}}\mathcal{F}(\phi)\|^2_{H^{-1}(D_\zeta)}=\sup_{\eta\in H^1_0(D_\zeta)\setminus\{0\}}\frac{|D_{\phi_>,{\rm red}}\mathcal{F}(\phi)[\eta]|^2}{\|\eta\|^2_{H^1_0(D_\zeta)}}=
4\sum_{j=2}^\infty\frac{(\lambda_j(\zeta)-(d-1))^2}{1+\lambda_j(\zeta)}c_j^2.\ee
where the supremum above is obtained for $$\eta:=\sum_{j=2}^\infty\frac{\lambda_j(\zeta)-(d-1)}{1+\lambda_j(\zeta)}c_j\varphi_j^{D_\zeta}.$$
On the other hand \be\label{eq:Qzeta-eq}\mathcal{Q}_\zeta(\phi_>)=\sum_{j=2}^\infty (\lambda_j(\zeta)-(d-1))c_j^2.\ee
Now we observe that the function $\lambda\mapsto\frac{1+\lambda}{\lambda-(d-1)}$ is decreasing for $\lambda>d-1$, then, for every $j\ge2$, we have
 $$\lambda_j(\zeta)-(d-1)=\frac{1+\lambda_j(\zeta)}{\lambda_j(\zeta)-(d-1)}\frac{(\lambda_j(\zeta)-(d-1))^2}{1+\lambda_j(\zeta)}\le \frac{1+\lambda_2(\zeta)}{\lambda_2(\zeta)-(d-1)}\frac{(\lambda_j(\zeta)-(d-1))^2}{1+\lambda_j(\zeta)}.$$
Using the spectral gap for $\lambda_2(\zeta)$ in \eqref{eq:gap-eigenvalue}, we have $$\lambda_j(\zeta)-(d-1)\le C\frac{(\lambda_j(\zeta)-(d-1))^2}{1+\lambda_j(\zeta)},$$ for some $C=C(d,b)>0$. Therefore, by combining \eqref{eq:Qzeta-eq0} and \eqref{eq:Qzeta-eq}, we get
$$0\le \mathcal{Q}_\zeta(\phi_>)\le C\|D_{\phi_>,{\rm red}}\mathcal{F}(\phi)\|_{H^{-1}(D_\zeta)}^2,$$
concluding the proof.
\end{proof}
\subsection{The first mode} As a consequence of the finite dimensional \L ojasiewicz inequality, we prove the following estimate for the first mode.
\begin{lemma}\label{lemma:lemma2-loj}
    Under the assumptions of \cref{thm:loj-ineq}, there exist constants $\theta\in(0,\frac12]$ and $C>0$ depending only on $d$ and $b$ such that \be\label{eq:claim-kernel}|\mathcal{G}(\zeta,\sigma)|^{1-\theta}\le C\|\nabla_{\zeta,{\rm red}}\mathcal{F}(\phi)\|_{L^2(\partial D)}+C|\partial_{\sigma,{\rm red}}\mathcal{F}(\phi)|.\ee 
\end{lemma}
\begin{proof}
Let $\delta_0$ be chosen sufficiently small, such that $\delta_0\le\overline\delta$, where $\overline\delta$ is the constant in \cref{lemma2.3}. We also assume that $\delta_0$ is sufficiently small in order to apply the Lyapunov-Schmidt reduction in \cref{lemma:Lyapunov-Schmidt}.

For $\xi=(\zeta,\sigma)$, we decompose $\xi=\xi_K+\widetilde\xi^\perp,$ where $\xi_K\in K$ and $\widetilde\xi^\perp\in K^\perp$. 
We recall that $K:=\ker \mathcal{L}$ is the kernel of the second variation defined in \cref{def:linerized-operators}.
If $Y$ is the Lyapunov-Schmidt map in \cref{lemma:Lyapunov-Schmidt}, setting $\xi^\perp:=\widetilde \xi^\perp-Y(\xi_K)\in K^\perp$, we obtain the decomposition $$\xi=\xi_0+\xi^\perp,\quad\text{where}\quad \xi_0=\xi_K+Y(\xi_K).$$
By \cref{lemma:Lyapunov-Schmidt}, we have that $\nabla\mathcal{G}(\xi_0)[\xi^\perp]=0$.
Applying the identity $$f(1)=f(0)+f'(0)+\frac12A+\int_0^1(1-t)(f''(t)-A)\,dt,\quad A\in\R$$ to $f(t):=\mathcal{G}(\xi_0+t\xi^\perp)$ and $A:=\delta^2\mathcal{G}(0)[\xi^\perp,\xi^\perp]$, we deduce that \be\label{eq:threeterms}\mathcal{G}(\xi)=\mathcal{G}(\xi_0)+\frac12\delta^2\mathcal{G}(0,0)[\xi^\perp,\xi^\perp]+R(\xi),\ee where $$R(\xi):=\int_0^1(1-t) \Big(\delta^2\mathcal{G}(\xi_0+t\xi^\perp)-\delta^2\mathcal{G}(0)\Big)[\xi^\perp,\xi^\perp]\,dt.$$
Therefore, in order to prove \eqref{eq:claim-kernel}, we estimate the three terms in \eqref{eq:threeterms}. 

\smallskip
\noindent\textit{Step 1.}
In the first step we prove the preliminary estimate \be\label{eq:step1} \|\xi^\perp\|_{X}\le C \|\nabla_{\zeta,{\rm red}}\mathcal{F}(\phi)\|_{L^2(\partial D)}+C|\partial_{\sigma,{\rm red}}\mathcal{F}(\phi)|,\ee for some $C=C(d,b)>0$.

In order to prove \eqref{eq:step1}, we first show that, denoting by $P_K$ and $P_{K^\perp}$ the projections in $L^2(\partial D)\oplus\R$ on $K$ and $K^\perp$, for every $f\in X'$ we have \be\label{eq:(v)}\|P_K f\|_{L^2(\partial D)\oplus\R}\le C\|f\|_{X'}\quad\text{and}\quad \|P_{K^\perp} f\|_{X'}\le C\|f\|_{X'}.\ee Indeed, denoting by $\Xi_1,\ldots,\Xi_N$ an orthonormal basis of $K$, the first inequality in \eqref{eq:(v)} follows by writing $P_Kf=\sum_{j=1}^N\langle f,\Xi_j\rangle\Xi_j$ and using that $|\langle f,\Xi_j\rangle |\le \|f\|_{X'}\|\Xi_j\|_{X}$.
On the other hand, the second inequality in \eqref{eq:(v)} is a consequence of the first one and the continuous embedding $L^2(\partial D)\hookrightarrow H^{-1/2}(\partial D)$. This concludes the proof of \eqref{eq:(v)}.

Now we continue the proof of \eqref{eq:step1}.
By \cref{lemma:Lyapunov-Schmidt}, $P_{K^\perp}\nabla\mathcal{G}(\xi_0)=0$, then \bea P_{K^\perp}\nabla\mathcal{G}(\xi)&=P_{K^\perp}\nabla\mathcal{G}(\xi)-P_{K^\perp}\nabla\mathcal{G}(\xi_0)=\int_0^1P_{K^\perp}\delta^2\mathcal{G}(\xi_0+t\xi^\perp)[\xi^\perp]\,dt\\&=\mathcal{L}^\perp\xi^\perp+\int_0^1P_{K^\perp}\Big(\delta^2\mathcal{G}(\xi_0+t\xi^\perp)-\delta^2\mathcal{G}(0)\Big)[\xi^\perp]\,dt,\eea
where we used that $P_{K^\perp}\mathcal{L}\xi^\perp=\mathcal{L}^\perp \xi^\perp$.
Using \cref{lemma2.3} (ii) and \eqref{eq:(v)} together with the smallness assumption on $\xi$, we have that $$\left\|\int_0^1P_{K^\perp}\Big(\delta^2\mathcal{G}(\xi_0+t\xi^\perp)-\delta^2\mathcal{G}(0)\Big)[\xi^\perp]\,dt\right\|_{X'}\le C\omega(\delta_0)\|\xi^\perp\|_{X}.$$ Combining the previous estimates with \cref{lemma2.3} (iv), we have 
\be\label{eq:prima-loj}\|\xi^\perp\|_{X}\le C\|\mathcal{L}^\perp\xi^\perp\|_{X'}\le C\|P_{K^\perp}\nabla\mathcal{G}(\xi)\|_{X'}+C\omega(\delta_0)\|\xi^\perp\|_{X}.\ee Therefore, by choosing $\delta_0$ small enough and using \eqref{eq:(v)}, we have \be\label{eq:seconda-loj} \|\xi^\perp\|_{X}\le C\|P_{K^\perp}\nabla\mathcal{G}(\xi)\|_{X'}\le C\|\nabla\mathcal{G}(\xi)\|_{X'}\le C\|\nabla\mathcal{G}(\xi)\|_{L^2(\partial D)\oplus \R},\ee
where in the last inequality we used that $L^2(\partial D)\hookrightarrow H^{-1/2}(\partial D)$.
Finally, by definition of $\nabla_{\zeta,{\rm red}}\mathcal{F}(\phi)$ and $\partial_{\sigma,{\rm red}}\mathcal{F}(\phi)$ (see \cref{subsection-important}), if $\phi$ is identified by the triple $(\zeta,\sigma,\phi_>)$, we have
\be\label{eq:2gradients}\nabla_{\zeta,{\rm red}}\mathcal{F}(\phi)=\nabla_\zeta\mathcal{G}(\zeta,\sigma)\quad\text{and}\quad \partial_{\sigma,{\rm red}}\mathcal{F}(\phi)=\partial_\sigma\mathcal{G}(\zeta,\sigma).\ee Then we conclude \eqref{eq:step1}.

\smallskip
\noindent\textit{Step 2.}
In order to estimate the first term in \eqref{eq:threeterms}, we use the finite dimensional \L ojasiewicz inequality to prove that \be\label{eq:step2}|\mathcal{G}(\xi_0)|^{1-\theta}\le C\|\nabla_{\zeta,{\rm red}}\mathcal{F}(\phi)\|_{L^2(\partial D)}+C|\partial_{\sigma,{\rm red}}\mathcal{F}(\phi)|,\ee
for some $\theta=\theta(d,b)\in(0,\frac12]$ and $C=C(d,b)>0$. 
 
Denoting by $\Xi_1,\ldots,\Xi_N$ an orthonormal basis of $K$, we can write  $$\xi_0=\Xi_{\mu_0}+Y(\Xi_{\mu_0}), \quad\text{where}\quad \Xi_{\mu_0}:=\sum_{j=1}^N(\mu_0)_j\Xi_j,$$ for some $\mu_0\in\R^N$. We consider the reduced functional $G$ defined as $G(\mu):=\mathcal{G}(\Xi_{\mu}+Y(\Xi_{\mu})),$ so that $G(\mu_0)=\mathcal{G}(\xi_0)$.
Since $G$ is analytic, we can apply the \L ojasiewicz inequality \cite{lojasiewicz} in a neighborhood of $0\in\R^N$. Then, if $\delta_0 $ is chosen small enough, we have
\be\label{eq:loj-finite-dim}|G(\mu_0)|^{1-\theta}\le C|\nabla G(\mu_0)|,\ee for some $\theta=\theta(d,b)\in(0,\frac12]$ and $C=C(d,b)>0$. 

We compute, for $j=1,\ldots,N$ $$\partial_{\mu_j}G(\mu_0)=\delta\mathcal{G}(\xi_0)[\Xi_j+D Y(\Xi_{\mu_0})[\Xi_j]]=\delta\mathcal{G}(\xi_0)[\Xi_j],$$ where in the last equality we used that $P_{K^\perp}\nabla \mathcal{G}(\xi_0)=0$ and $D Y(\Xi_{\mu_0})[\Xi_j]\in K^\perp$, by \cref{lemma:Lyapunov-Schmidt}. 
Then, using again that $P_{K^\perp}\nabla\mathcal{G}(\xi_0)=0$, we have $$|\nabla_{\mu}G(\mu_0)|=\|P_K\nabla\mathcal{G}(\xi_0)\|_{L^2(\partial D)\oplus\R}=\|\nabla\mathcal{G}(\xi_0)\|_{L^2(\partial D)\oplus\R}.$$
Therefore, by \eqref{eq:loj-finite-dim}, we obtain \be\label{eq:intermediate-step2}| \mathcal{G}(\xi_0)|^{1-\theta}\le C\|\nabla\mathcal{G}(\xi_0)\|_{L^2(\partial D)\oplus\R}.\ee

Now we need to replace the gradient in the right-hand side with the reduced gradient of $\mathcal{F}$, and in particular we need to replace $\xi_0$ with $\xi$.
We first observe that, by \cref{lemma:Lyapunov-Schmidt}
$$\nabla\mathcal{G}(\xi_0)=P_{K}\nabla\mathcal{G}(\xi_0)=P_K\nabla\mathcal{G}(\xi)+P_K\Big(\nabla\mathcal{G}(\xi_0)-\nabla\mathcal{G}(\xi)\Big).$$ Moreover, by \eqref{eq:(v)}, we get
$$\left\|P_K\Big(\nabla\mathcal{G}(\xi_0)-\nabla\mathcal{G}(\xi)\Big)\right\|_{L^2(\partial D)\oplus\R}\le C\|\nabla\mathcal{G}(\xi_0)-\nabla\mathcal{G}(\xi)\|_{X'},$$ and, by \cref{lemma2.3} (ii) \bea \left\|\nabla\mathcal{G}(\xi_0)-\nabla\mathcal{G}(\xi)\right\|_{X'}&=\left\|\int_0^1\delta^2\mathcal{G}(\xi_0+t\xi^\perp)[\xi^\perp]\,dt\right\|_{X'}\le C\|\xi^\perp\|_{X}.\eea 
 Then, combining the previous estimate and using \eqref{eq:step1}, we have \bea\label{eq:eq1-step2}\begin{aligned}\|\nabla\mathcal{G}(\xi_0)\|_{L^2(\partial D)\oplus\R}&\le \|P_{K}\nabla\mathcal{G}(\xi)\|_{L^2(\partial D)\oplus\R}+C\|\xi^\perp\|_{X}\\&\le \|\nabla\mathcal{G}(\xi)\|_{L^2(\partial D)\oplus\R}+ C \|\nabla_{\zeta,{\rm red}}\mathcal{F}(\phi)\|_{L^2(\partial D)}+C|\partial_{\sigma,{\rm red}}\mathcal{F}(\phi)|\\&\le C \|\nabla_{\zeta,{\rm red}}\mathcal{F}(\phi)\|_{L^2(\partial D)}+C|\partial_{\sigma,{\rm red}}\mathcal{F}(\phi)|,\end{aligned}\eea where in the last inequality we used \eqref{eq:2gradients}. Using this estimate in \eqref{eq:intermediate-step2}, we conclude \eqref{eq:step2}.

\smallskip
\noindent\textit{Step 3.}
Next we estimate the second term in \eqref{eq:threeterms}, proving that \be\label{eq:step3}|\delta^2\mathcal{G}(0)[\xi^\perp,\xi^\perp]|^{1/2}\le C\|\nabla_{\zeta,{\rm red}}\mathcal{F}(\phi)\|_{L^2(\partial D)}+C|\partial_{\sigma,{\rm red}}\mathcal{F}(\phi)|.\ee
Recalling the linearized operator $\mathcal{L}^\perp$ in \cref{lemma2.3} (iv), we have $$|\delta^2\mathcal{G}(0)[\xi^\perp,\xi^\perp]|=|\langle \mathcal{L}^\perp\xi^\perp,\xi^\perp\rangle|\le \|\mathcal{L}^\perp \xi^\perp\|_{X'}\| \xi^\perp\|_{X}.$$ Moreover, in the proof of Step 1 (see \eqref{eq:prima-loj} and \eqref{eq:seconda-loj}) we proved that $$\| \xi^\perp\|_{X}\le C\|\mathcal{L}^\perp \xi^\perp\|_{X'}\le C\|\nabla_{\zeta,{\rm red}}\mathcal{F}(\phi)\|_{L^2(\partial D)}+C|\partial_{\sigma,{\rm red}}\mathcal{F}(\phi)|.$$ Then \eqref{eq:step3} follows by combining the last two estimates.

\smallskip
\noindent\textit{Step 4.}
Now we estimate the third term in \eqref{eq:threeterms}, proving that \be\label{eq:step4}|R(\xi)|^{1/2}\le C\|\nabla_{\zeta,{\rm red}}\mathcal{F}(\phi)\|_{L^2(\partial D)}+C|\partial_{\sigma,{\rm red}}\mathcal{F}(\phi)|.\ee
This estimate immediately follows by combining \cref{lemma2.3} (ii) with \eqref{eq:step1}.
 
 \smallskip
\noindent\textit{Step 5.} In order to conclude the proof, we first observe the trivial inequality \be\label{eq:trivial-inequ}|t+s|^{1-\theta}\le |t|^{1-\theta}+|s|^{1-\theta}\quad\text{for every }t,s\in\R, \ \theta\in\left(0,1/2\right]. \ee Moreover, by \eqref{eq:step3} and \eqref{eq:step4}, we have $$|\delta^2\mathcal{G}(0,0)[\xi^\perp,\xi^\perp]|+|R(\xi)|\le C\|\nabla_{\zeta,{\rm red}}\mathcal{F}(\phi)\|_{L^2(\partial D)}+C|\partial_{\sigma,{\rm red}}\mathcal{F}(\phi)|\le C,$$ for some $C=C(d,b)>0$, if $\delta_0$ is chosen sufficiently small.
Therefore, by \eqref{eq:threeterms} and \eqref{eq:step2}, we have \bea |\mathcal{G}(\xi)|^{1-\theta}&\le |\mathcal{G}(\xi_0)|^{1-\theta}+\left|\frac12\delta^2\mathcal{G}(0,0)[\xi^\perp,\xi^\perp]\right|^{1-\theta}+|R(\xi)|^{1-\theta}\\&\le |\mathcal{G}(\xi_0)|^{1-\theta}+C|\delta^2\mathcal{G}(0,0)[\xi^\perp,\xi^\perp]|^{1/2}+C|R(\xi)|^{1/2}\\&\le 
C\|\nabla_{\zeta,{\rm red}}\mathcal{F}(\phi)\|_{L^2(\partial D)}+C|\partial_{\sigma,{\rm red}}\mathcal{F}(\phi)|,\eea for some $C=C(d,b)>0$, concluding the proof.
 \end{proof}

\subsection{Proof of the \L ojasiewicz inequality}
Now we can prove the \L ojasiewicz inequality for the spherical energy $\mathcal{F}$ in \cref{thm:loj-ineq}.
\begin{proof}[Proof of \cref{thm:loj-ineq}]
    We recall the decomposition of the energy in \eqref{eq:decomposition}, namely $\mathcal{F}(\phi)-\mathcal{F}(b)=\mathcal{G}(\zeta,\sigma)+\mathcal{Q}_\zeta(\phi_>)$. In particular, since $|\mathcal{F}(\phi)-\mathcal{F}(b)|\le C_0$ by \eqref{eq:closass} and $|\mathcal{G}(\zeta,\sigma)|\le C$, we have $$|Q_\zeta(\phi_>)|\le C(C_0+1),$$ for some $C=C(d,b)>0$. Then, by combining the estimates in \cref{lemma:lemma1-loj} and \cref{lemma:lemma2-loj} with the inequality in \eqref{eq:trivial-inequ}, we get \bea\label{eq:fineloj}\begin{aligned}|\mathcal{F}(\phi)-\mathcal{F}(b)|^{1-\theta}&\le |\mathcal{G}(\zeta,\sigma)|^{1-\theta}+|Q_\zeta(\phi_>)|^{1-\theta}\le |\mathcal{G}(\zeta,\sigma)|^{1-\theta}+C|Q_\zeta(\phi_>)|^{1/2} \\&\le C\|\nabla_{\zeta,{\rm red}}\mathcal{F}(\phi)\|_{L^2(\partial D)}+C|\partial_{\sigma,{\rm red}}\mathcal{F}(\phi)|+C\|D_{\phi_>,\rm red}\mathcal{F}(\phi)\|_{H^{-1}(D_\zeta)},\end{aligned}\eea for some $C=C(d,b,C_0)>0$, concluding the proof of the \L ojasiewicz inequality. 
\end{proof}
\begin{remark}\label{corollary:gamma12}
    We observe that if $b$ is integrable, then we can choose $\theta=\frac12$ in \cref{lemma:lemma2-loj}.
    Indeed, in this case, following the notation in \cref{lemma:lemma2-loj}, we have $\mathcal{G}(\xi_0)=G(\mu_0)=0$, by \cref{lemma:adam-simon}. In particular, we can choose $\theta=\frac12$ in \eqref{eq:step2}. Then, the thesis follows by combining the estimates in \eqref{eq:step3} and \eqref{eq:step4}.
\end{remark}
\section{Radial control of the reduced gradient}\label{sec4}

The main result of this section is the following proposition, in which we prove that the reduced gradient of $\mathcal{F}$ is controlled by the radial derivative $r\partial_ru_r=\nabla u_r\cdot x-u_r$. 
\begin{proposition}\label{prop:propfundamental}
    There exists $\delta>0$ depending only on $d$, $b$ and $L$ such that the following holds. 
    Suppose that $u$ satisfies the assumptions of \cref{thm:epiperimetric-inequality}, for some $\rho\in(0,1]$ and $\delta>0$. We denote by $\phi_r(\theta):=r^{-1} u(r,\theta)\in H^1(\partial B_1)$. Then $$\int_{5\rho/8}^{7\rho/8}\frac1r\|\nabla_{{\rm red}}\mathcal{F}(\phi_r)\|^2\,dr\le C\int_{\rho/2}^{\rho}\frac1r\int_{\partial B_1}(\nabla u_r\cdot x-u_r)^2\,d\HH^{d-1}\,dr,$$ for some $C>0$ depending only on $d$ and $b$.
\end{proposition}
\subsection{Smooth parametrization lemma}
We first prove that if $u$ is sufficiently close to $b$, then the free boundary of $u$ is a $C^{2,\alpha}$ (actually smooth) graph in an annulus.
\begin{lemma}\label{lemma:smooth-parametrization-lemma}
    For every $\widetilde\delta>0$ there exists $\delta>0$ depending only on $\widetilde\delta$, $d$, $b$ and $L$ such that the following holds. 
    
    Suppose that $u$ satisfies the assumptions of \cref{thm:epiperimetric-inequality}, for some $\rho\in(0,1]$ and $\delta>0$. Then, there exists $\zeta(r,\cdot):\partial D\to\R,$ with $\zeta\in C^{2,\alpha}((9\rho/16,15\rho/16)\times \partial D)$ such that, for every $r\in(9\rho/16,15\rho/16)$, we have \be\label{eq:control-zeta}\partial\Omega_{u_r}\cap \partial B_1=\text{graph}_{\partial D}(\zeta(r,\cdot)),\ D_{\zeta(r,\cdot)}=\Omega_{u_r}\cap \partial B_1\quad\text{and}\quad  \|\zeta(\rho\,\cdot,\cdot)\|_{C^{2,\alpha}((9/16,15/16)\times\partial D)}\le\widetilde \delta.\ee 
\end{lemma}

\begin{proof}
    After a rescaling we can assume $\rho=1$. We suppose by contradiction that there exist a sequence of variational solutions $\{u^j\}\subset H^1(B_1)$, with $0\in\partial\Omega_{u^j}$, and numbers $\delta_j\to0^+$ such that \be\label{eq:contr-ass-smooth}\|u^j-b\|_{L^2(\partial B_1)}\le \delta_j,\quad |W(u^j)-W(b)|\le \delta_j,\quad \Theta(u^j,0)\ge W(b),\ee but the conclusion \eqref{eq:control-zeta} for $\rho=1$ does not hold. 
     
      By \cref{lemma:compactness-kw}, the sequence $\{u^j\}$ converges strongly in $H^1(B_{1})\cap C^{0,\alpha}(\overline B_1)$ to some variational solution $u^\infty$, with $u^\infty=b$ on $\partial B_1$. In particular, $u^\infty\not\equiv0$. Then, again by \cref{lemma:compactness-kw}, we have that $\Omega_{u^j}\to \Omega_{u^\infty}$, $\partial\Omega_{u^j}\to \partial\Omega_{u^\infty}$ locally in the Hausdorff distance in $B_1$, and $\chi_{\Omega_{u^j}}\to \chi_{\Omega_{u^\infty}}$ in $L^1_{\text{loc}}(B_1)$. In particular, \be\label{eq:conv-rr}W(u^j_r)\to W(u^\infty_r)\quad\text{for every }r\in(0,1).\ee By \eqref{eq:contr-ass-smooth} and the monotonicity formula \cref{prop:monotonicity-formula}, we have that $$W(u^j)\ge W(u^j_r)\ge W(b)\quad\text{for every }r\in(0,1).$$ Then, passing to the limit as $j\to+\infty$ and using that $W(u^j)\to W(b)$ and \eqref{eq:conv-rr}, we get 
      $$W(u^\infty_r)\equiv W(b)\quad\text{for every }r\in(0,1).$$ Thus, by \cref{prop:monotonicity-formula}, $u^\infty$ is $1$-homogeneous. Since $u^\infty=b$ on $\partial B_1$, we have that $u^\infty\equiv b$ in $B_1$. 

Now we claim that, for every $j$ sufficiently large, we have \be\label{eq:claim-j}\Sigma^H_{u^j}\cap (B_{31/32}\setminus \overline B_{1/2})=\emptyset,\ee where $\Sigma^H_{u^j}$ is defined in \eqref{eq:sigmaH}. Suppose by contradiction that there exists a sequence $\{x^j\}\subset \Sigma^H_{u^j} \cap (B_{31/32}\setminus \overline B_{1/2})$. Using the Hausdorff convergence of the free boundaries, we have $x^j\to x^\infty\in \partial\Omega_b\cap (\overline B_{31/32}\setminus B_{1/2})$ up to a subsequence. Moreover, since $b$ is smooth up to the free boundary outside the origin, then $x^\infty$ is a regular point. Therefore, $\Theta(b,x^\infty)=\omega_d/2$ and thus, by definition of the function $\Theta$, we have $$W(b_{r,x^\infty})<3\omega_d/4\quad\text{for some }r\in(0,1/100).$$ On the other hand, since $x^j\in \Sigma^H_{u^j}$, then $\Theta(u^j,x^j)=\omega_d$, and thus $$W(u^j_{r,x^j})\ge \omega_d,$$ by the monotonicity formula \cref{prop:monotonicity-formula}.
However, by the strong $H^1(B_1)$-convergence and uniform convergence of $u^j$ to $b$, the $L^1_{\text{loc}}(B_1)$-convergence of $\chi_{\Omega_{u^j}}$ to $\chi_{\Omega_{b}}$ and the convergence of $x^j$ to $x^\infty$, we get $W(u^j_{r,x^j}) \to W(b_{r,x^\infty})$ as $j\to+\infty$, a contradiction. 
This concludes the proof of \eqref{eq:claim-j}.

By \eqref{eq:claim-j} and \cref{prop:kriventsov-weiss} (iii), we obtain that $u^j$ is a viscosity solution in $B_{31/32}\setminus \overline B_{1/2}$, for every $j$ large enough. By combining the uniform convergence of $u^j$ to $b$ and the Hausdorff convergence of the positivity sets and the free boundaries, and using that $b$ is smooth up to the free boundary outside the origin, we can apply the improvement of flatness theorem in \cite{DeSilva:FreeBdRegularityOnePhase} together with the hodograph map argument (see \cite[Theorem 2]{KinderlehrerNirenberg1977:AnalyticFreeBd}) to the functions $u^j$, for $j$ large enough. Then, we deduce that $\partial\Omega_{u^j}\cap (B_{15/16}\setminus B_{9/16})$ is a $C^{2,\alpha}$ graph, with $C^{2,\alpha}$ norm bounded by a small constant, which is a contradiction. 
\end{proof}
\subsection{A reduced gradient estimate}
In order to prove \cref{prop:propfundamental}, in the next lemma we show a preliminary estimate of the reduced gradient. 

\begin{lemma}\label{proposition:corollary}
    Under the assumptions of \cref{prop:propfundamental}, for every $r\in I:=(9\rho/16,15\rho/16)$ there exists a graph $\zeta_r\in C^{2,\alpha}(\partial D)$, with uniformly small $C^{2,\alpha}$ norm, such that \bea \|\nabla_{{\rm red}}\mathcal{F}(\phi_r)\|^2\le C\|r\partial_r \phi_r\|^2_{L^2(\partial D_{\zeta_r})}+C\|r^2\partial_{rr}\phi_r\|^2_{L^{2}(D_{\zeta_r})}+C\|r\partial_r\phi_r\|^2_{L^{2}(D_{\zeta_r})},\eea for some constant $C>0$ depending only on $d$ and $b$.
\end{lemma} 
First we prove the following lemma.
 \begin{lemma}\label{lemma:gradient-explicitly-bounds}
    Let $\phi\in H^1(\partial B_1)$ be a non-negative trace sufficiently close to $b$ in $L^2(\partial B_1)$, and suppose that $\partial\Omega_\phi\cap \partial B_1$ is a small $C^{2,\alpha}$ graph over $\partial D$, namely \eqref{eq:graph-assumption} holds. We also suppose that $\phi\in C^1(\overline D_\zeta)$.  
    Then, we have the estimates \be\label{eq:lemma-preliminary1}\|\nabla_{\zeta,{\rm red}}\mathcal{F}(\phi)\|_{L^2(\partial D)}^2 \le C\|1-|\nabla_\theta \phi|^2\|^2_{L^2(\partial D_\zeta)}+C\|\partial_{\nu_\zeta}\phi_>\|^2_{L^2(\partial D_\zeta)},\ee and
\be\label{eq:lemma-preliminary2}|\partial_{\sigma,{\rm red}}\mathcal{F}(\phi)|^2+\|D_{\phi_>,{\rm red}}\mathcal{F}(\phi)\|_{H^{-1}(D_\zeta)}^2\le C\|\Delta_\theta \phi+(d-1)\phi\|^2_{H^{-1}(D_\zeta)},\ee for some constant $C>0$ depending only on $d$ and $b$.
 \end{lemma}
  \begin{proof}
    First we prove \eqref{eq:lemma-preliminary1}. We observe that, by \eqref{eq:gradient-explicitly-zeta} \be\label{eq:combined-preliminary}\|\nabla_{\zeta,{\rm red}}\mathcal{F}(\phi)\|_{L^2(\partial D)}^2 \le C\|1-\kappa^2(\partial_{\nu_\zeta} \varphi_1^{D_\zeta})^2\|^2_{L^2(\partial D_\zeta)}.\ee Moreover, since $\phi=\kappa\varphi_1^{D_\zeta}+\phi_>$, where $\kappa:=(\kappa_0^2+\sigma)^{1/2}$, we have that \bea 1-\kappa^2(\partial_{\nu_\zeta}\varphi_1^{D_\zeta})^2&=1-(\partial_{\nu_\zeta} \phi-\partial_{\nu_\zeta} \phi_>)^2=1-|\nabla_\theta \phi|^2+2\partial_{\nu_\zeta}\phi\partial_{\nu_\zeta}\phi_>-(\partial_{\nu_\zeta}\phi_>)^2\\&=1-|\nabla_\theta \phi|^2+2\kappa\partial_{\nu_\zeta}\varphi_1^{D_\zeta}\partial_{\nu_\zeta}\phi_>+(\partial_{\nu_\zeta}\phi_>)^2,\eea where the quantities above are well-defined on $\partial D_\zeta$ since $\phi\in C^1(\overline D_\zeta)$.
    We also observe that, since $\zeta$ is close to $0$ in $C^{2,\alpha}$, we have $|\kappa\partial_{\nu_\zeta}\varphi_1^{D_\zeta}|\le C$. In order to conclude the proof of \eqref{eq:lemma-preliminary1}, we distinguish two cases. If $|\partial_{\nu_\zeta}\phi_>|\le 1$, then $$(1-\kappa^2(\partial_{\nu_\zeta}\varphi_1^{D_\zeta})^2)^2\le C(1-|\nabla_\theta \phi|^2)^2+C(\partial_{\nu_\zeta}\phi_>)^2.$$ Otherwise, if $|\partial_{\nu_\zeta}\phi_>|> 1$, then $$(1-\kappa^2(\partial_{\nu_\zeta}\varphi_1^{D_\zeta})^2)^2\le C\le C(\partial_{\nu_\zeta}\phi_>)^2\le C(1-|\nabla_\theta \phi|^2)^2+C(\partial_{\nu_\zeta}\phi_>)^2.$$ 
    Combining the above two estimates with \eqref{eq:combined-preliminary}, we conclude \eqref{eq:lemma-preliminary1}.

   Next we prove \eqref{eq:lemma-preliminary2}.
   Since $\sigma$ is small and $\kappa_0$ is away from $0$, then by \eqref{eq:gradient-explicitly-other}, we have \be\label{eq:one-estimate}|\partial_{\sigma,{\rm red}}\mathcal{F}(\phi)|^2= (\lambda_1(\zeta)-(d-1))^2\le C\kappa^2(\lambda_1(\zeta)-(d-1))^2.\ee For $f\in H^{-1}(D_\zeta)$ with expansion $f=\sum_{j=1}^\infty a_j\varphi_j^{D_\zeta}\in H^{-1}(D_\zeta)$, with $\{a_j\}\subset\R$, we have \bea\label{eq:dopouso}\|f\|^2_{H^{-1}(D_\zeta)}=\sup_{\eta\in H^1_0(D_\zeta)\setminus\{0\}}\frac{|\langle f,\eta\rangle |^2}{\|\eta\|^2_{H^1_0(D_\zeta)}}=\sum_{j=1}^\infty \frac{a_j^2}{1+\lambda_j(\zeta)},\eea where the supremum above is obtained for $\eta:=\sum_{j=1}^\infty\frac{a_j}{1+\lambda_j(\zeta)}\varphi_j^{D_\zeta}.$ Applying this identity to 
    $$-\Delta_\theta \phi-(d-1)\phi=\kappa(\lambda_1(\zeta)-(d-1))\varphi_1^{D_\zeta}+\sum_{j=2}^\infty(\lambda_j(\zeta)-(d-1))c_j\varphi_j^{D_\zeta},$$ 
    and using \eqref{eq:gradient-explicitly-other} and \eqref{eq:one-estimate}, we obtain 
    \bea \|\Delta_\theta \phi+(d-1)\phi\|^2_{H^{-1}(D_\zeta)}&=\frac{\kappa^2(\lambda_1(\zeta)-(d-1))^2}{1+\lambda_1(\zeta)}+\sum_{j=2}^\infty\frac{(\lambda_j(\zeta)-(d-1))^2}{1+\lambda_j(\zeta)}c_j^2\\&=\frac{\kappa^2(\lambda_1(\zeta)-(d-1))^2}{1+\lambda_1(\zeta)}+\|\Delta_\theta\phi_>+(d-1)\phi_>\|^2_{H^{-1}(D_\zeta)}\\&\ge C^{-1}|\partial_{\sigma,{\rm red}}\mathcal{F}(\phi)|^2+4^{-1}\|D_{\phi_>,{\rm red}} \mathcal{F}(\phi)\|^2_{H^{-1}(D_\zeta)}.
    \eea 
    The last estimate concludes the proof of \eqref{eq:lemma-preliminary2}.
\end{proof}
\begin{proof}[Proof of \cref{proposition:corollary}]
    By \cref{lemma:smooth-parametrization-lemma}, $u$ is a classical solution of the one-phase problem in $A_\rho:=B_{15\rho/16}\setminus B_{9\rho/16}$, namely $u\in C^{2,\alpha}(\overline\Omega_u\cap A_\rho)$. Then, $u$ solves $\Delta u=0$ in $\Omega_u\cap A_\rho$ and $|\nabla u|^2=1$ on $\partial\Omega_u\cap A_\rho$. Therefore we have \be\label{eq:equation}\partial_{rr}u+\frac{d-1}r\partial_ru+\frac{1}{r^2}\Delta_\theta u=0\quad\text{in } \Omega_u\cap A_\rho,\quad |\partial_ru|^2+\frac{1}{r^2}|\nabla_\theta u|^2=1\quad\text{on }\partial\Omega_u\cap A_\rho.\ee
    We write $u(r,\theta)=r\phi_r(\theta)$ and we decompose $\phi_r$ according to the notation in \cref{subsection-important}. More precisely, for $\zeta_r:=\zeta(r,\cdot)$, we write $\phi_r=(\zeta_r,\sigma_r,(\phi_r)_>)$ for every $r\in I:=(9\rho/16,15\rho/16)$, namely $$\phi_r=(\kappa_0^2+\sigma_r)^{1/2}\varphi_1^{D_{\zeta_r}}+(\phi_r)_>.$$ Then, since $\phi_r=0$ on $\partial D_{\zeta_r}$, \eqref{eq:equation} becomes for every $r\in I$
    \be\label{eq:computationabove} r\partial_{rr} \phi_r+(d+1)\partial_r\phi_r+\frac{1}{r}(\Delta_\theta\phi_r+(d-1)\phi_r)=0\quad\text{in } D_{\zeta_r},\quad r^2|\partial_r\phi_r|^2+|\nabla_\theta \phi_r|^2=1\quad\text{on } \partial D_{\zeta_r}.\ee 
By \cref{lemma:smooth-parametrization-lemma} and \eqref{eq:inequality-with-log}, after choosing $\delta>0$ sufficiently small, both $\|\zeta_r\|_{C^{2,\alpha}(\partial D)}$ and $\|\phi_r-b\|_{L^2(\partial B_1)}$ are sufficiently small, for every $r\in I$.
Therefore, we can apply \cref{lemma:gradient-explicitly-bounds} to deduce
\be\label{eq:dausaredopo}\begin{aligned} \|\nabla_{{\rm red}}\mathcal{F}(\phi_r)\|^2&\le  C\|1-|\nabla_\theta \phi_r|^2\|^2_{L^2(\partial D_{\zeta_r})}+C\|\partial_{\nu_{\zeta_r}}(\phi_r)_>\|^2_{L^2(\partial D_{\zeta_r})}\\&\qquad+C\|\Delta_\theta \phi_r+(d-1)\phi_r\|^2_{H^{-1}(D_{\zeta_r})},
\end{aligned}\ee for some $C=C(d,b)>0$.

We also observe that, since $D_{\zeta_r}$ are uniformly $C^{2,\alpha}$, for every $r\in I$, we have the trace inequality \bea \|\partial_{\nu_{\zeta_r}}(\phi_r)_>\|_{L^2(\partial D_{\zeta_r})}&\le C\|(\phi_r)_>\|_{H^2( D_{\zeta_r})}\\&\le C\|\Delta_\theta (\phi_r)_>+(d-1)(\phi_r)_>\|_{L^2( D_{\zeta_r})}+C\|(\phi_r)_>\|_{L^2( D_{\zeta_r})},\eea where in the last inequality we used elliptic estimates on uniformly $C^{2,\alpha}$ domains.
Using the gap \eqref{eq:gap-eigenvalue}, we also deduce that \bea \|(\phi_r)_>\|_{L^2( D_{\zeta_r})}\le C\|\Delta_\theta (\phi_r)_>+(d-1)(\phi_r)_>\|_{L^2( D_{\zeta_r})}.\eea 
Moreover, since $(\phi_r)_>$ is the projection of $\phi_r$ on the higher modes, and the projection on the higher modes commutes with $\Delta_\theta+(d-1)$, we have \bea \|\Delta_\theta (\phi_r)_>+(d-1)(\phi_r)_>\|_{L^2( D_{\zeta_r})}\le C\|\Delta_\theta \phi_r+(d-1)\phi_r\|_{L^2( D_{\zeta_r})}.\eea 
Combining the previous three estimates, we obtain that, 
$$\|\partial_{\nu_{\zeta_r}}(\phi_r)_>\|_{L^2(\partial D_{\zeta_r})}\le C\|\Delta_\theta \phi_r+(d-1)\phi_r\|_{L^2( D_{\zeta_r})}.$$ 
Therefore, by \eqref{eq:dausaredopo} and \eqref{eq:computationabove}, we obtain, for every $r\in I$,
\bea
\|\nabla_{{\rm red}}\mathcal{F}(\phi_r)\|^2&\le C\|1-|\nabla_\theta \phi_r|^2\|^2_{L^2(\partial D_{\zeta_r})}+C\|\Delta_\theta \phi_r+(d-1)\phi_r\|^2_{L^{2}(D_{\zeta_r})}
\\&\le C\|r\partial_r \phi_r\|^2_{L^2(\partial D_{\zeta_r})}+C\|r^2\partial_{rr}\phi_r\|^2_{L^{2}(D_{\zeta_r})}+C\|r\partial_r\phi_r\|^2_{L^{2}(D_{\zeta_r})},\eea where we used that $|r\partial_r \phi_r|\le1$ on $\partial D_{\zeta_r}$ by \eqref{eq:computationabove}, concluding the proof.
\end{proof}
\subsection{A Caccioppoli-type inequality} We show the following Caccioppoli-type inequality.
\begin{lemma}\label{lemma:estimateintime}
    Under the assumptions of \cref{prop:propfundamental}, let $q(t,\theta):=r\partial_r\phi_r(\theta)$ for $t:=\log r$. We also set $$\mathcal{I}':=(\log(5\rho/8),\log(7\rho/8))\subset \mathcal{I}:=(\log(9\rho/16),\log(15\rho/16)).$$ Then, we have $$\int_{\mathcal{I}'}\int_{D_t}(|\partial_t q|^2+|\nabla_\theta q|^2)\,d\HH^{d-1}\,dt\le C\int_{\mathcal{I}}\int_{D_t} q^2\,d\HH^{d-1}\,dt,$$ for some $C=C(d,b)>0$.
    \end{lemma}
\begin{proof}    
We define $\Phi(t,\theta):=\phi_r(\theta)=\phi_{e^t}(\theta)$, where $t:=\log r$. 
Then, by \eqref{eq:computationabove}, for every $t\in \mathcal{I}$, we have
\be\label{eq:computationabove-after} \partial_{tt}\Phi +d\partial_t\Phi+\Delta_\theta \Phi+(d-1)\Phi=0\quad\text{in } D_{t},\ee where $D_t:=D_{\zeta(e^t,\cdot)}$.

Now we set $q(t,\theta):=\partial_t\Phi(t,\theta)=r\partial_r\phi_r(\theta).$
Differentiating \eqref{eq:computationabove-after} with respect to $t$, we get \be\label{eq:computationabove-afterafter} \partial_{tt}q +d \partial_{t}q+\Delta_\theta q+(d-1)q=0\quad\text{in } \mathcal{D}:=\{(t,\theta):\ t\in \mathcal{I}, \ \theta\in D_t\}.
\ee
Now we claim that $q$ satisfies the free boundary condition
\be\label{eq:computationabove-afterafter2}\partial_{\overline\nu} q=\beta q\quad\text{on }\Sigma:=\{(t,\theta):\ t\in \mathcal{I}, \ \theta\in \partial D_t\},\quad \text{where}\quad \beta:=e^tH-\overline \nu_t,\ee where $\overline \nu=(\overline \nu_t,\overline \nu_\theta)\in\R\times T_\theta\mathbb{S}^{d-1}$ is the outward unit normal to $\Sigma$.

In order to prove \eqref{eq:computationabove-afterafter2}, we first observe that in the annulus $A_\rho:=B_{15\rho/16}\setminus B_{9\rho/16}$, the function $h:=\nabla u\cdot x-u$ solves \be\label{eq:h-solves}\partial_\nu h=Hh\quad\text{on }\partial\Omega_u\cap A_\rho,\ee where $\nu$ is the outward unit normal to $\partial \Omega_u$ and $H:=-\partial_{\nu\nu}u$ is the mean curvature of the free boundary $\partial\Omega_u$ in $A_\rho$. Indeed, as already seen in the proof of \cref{proposition:corollary}, $u$ is a classical solution in $A_\rho$. Differentiating $|\nabla u|=1$ on $\partial\Omega_u \cap A_\rho$ in the tangential directions, we get $D^2u\, \tau\cdot\nu=0$ for every tangential direction $\tau$. Then, $D^2u\,x\cdot \nu=(x\cdot\nu)\partial_{\nu\nu}u=-(x\cdot\nu)H$, and thus $$\partial_{\nu}h=\partial_{\nu}(\nabla u\cdot x)-\partial_{\nu }u=D^2u \,x\cdot \nu=-(x\cdot\nu)H.$$ We also observe that $h=\nabla u\cdot x=-(x\cdot\nu)$ on $\partial\Omega_u\cap A_\rho$.
This concludes the proof of \eqref{eq:h-solves}.

Next we proceed with the proof of the claim \eqref{eq:computationabove-afterafter2}. Since $q(t,\theta)=e^{-t}h(e^t\theta)$, if we write $\overline \nu=(\overline \nu_t,\overline \nu_\theta)$, so that $\partial_{\overline \nu}t=\overline \nu_t$ and $\partial_{\overline \nu}\theta=\overline \nu_\theta$, then \bea \partial_{\overline \nu}q=-e^{-t}h(e^t\theta)\partial_{\overline \nu}t+e^{-t}\nabla h(e^t\theta)\cdot \partial_{\overline \nu}(e^t\theta)=-q\overline \nu_t+\nabla h(e^t\theta)\cdot(\theta\overline\nu_t+\overline \nu_\theta)\eea
One can directly check that the vector $\theta\overline\nu_t+\overline \nu_\theta$ is orthogonal to every tangent vector to the free boundary, and has unit norm. In particular, $\nu=\pm(\theta\overline\nu_t+\overline \nu_\theta)$. Additionally, since $\nu$ and $\overline \nu$ are both chosen as outward unit normals to the corresponding domains, the sign is positive, and we get
$\nu=\theta\overline\nu_t+\overline \nu_\theta.$ Therefore, by \eqref{eq:h-solves}, we obtain
$$\partial_{\overline \nu}q=-q\overline \nu_t+\partial_\nu h=-q\overline \nu_t+Hh. $$ Since $h=e^tq$, we conclude the proof of the claim \eqref{eq:computationabove-afterafter2}.

We next claim that $\beta$ satisfies \be\label{eq:boundbeta}\|\beta\|_{L^\infty(\Sigma)}\le C,\ee for some $C=C(d,b)>0$. Indeed, we first recall that, by \eqref{eq:control-zeta}, after the rescaling $x=\rho y$, the free boundary has uniformly bounded $C^{2,\alpha}$ norm in the fixed annulus $A_1:=B_{15/16}\setminus B_{9/16}$. Thus, the mean curvature is bounded in $A_1$. Rescaling back, we obtain 
$\|H\|_{L^\infty (\partial\Omega_u\cap A_\rho)}\le C\rho^{-1}.$
Then, since $r$ and $\rho$ are comparable in $A_\rho$, we have $\|rH\|_{L^\infty(\partial\Omega_u\cap A_\rho)}\le C$. 
Since $r=e^t$, we deduce that $$\|e^tH\|_{L^\infty(\Sigma)}\le C.$$ Combining this estimate with the trivial bound $|\overline\nu_t|\le 1$, we conclude the proof of \eqref{eq:boundbeta}.

We now take an arbitrary $\eta=\eta(t)\in C_c^\infty(\mathcal{I})$. We test the weak formulation of \eqref{eq:computationabove-afterafter} with $\eta^2q$, obtaining the following Caccioppoli-type inequality
\be\label{eq:eq1} \int_{\mathcal{D}}\eta^2 (|\partial_t q|^2+|\nabla_\theta q|^2)\,d\HH^{d-1}\,dt\le C\int_{\mathcal{D}} (\eta^2+(\partial_t \eta)^2)q^2\,d\HH^{d-1}\,dt+C\left|\int_{\Sigma}\eta^2q\partial_{\overline\nu} q\,d\HH^{d-1}_\Sigma\right|,\ee for some $C=C(d,b)>0$.
By \eqref{eq:computationabove-afterafter2} and \eqref{eq:boundbeta}, we have that 
\bea \left|\int_{\Sigma}\eta^2q\partial_{\overline\nu} q\,d\HH^{d-1}_\Sigma\right|&\le C\left|\int_{\Sigma}\eta^2 q^2\,d\HH^{d-1}_\Sigma\right|\le C\left|\int_{\mathcal{I}}\int_{\partial D_t}\eta^2 q^2\,d\HH^{d-2}\,dt\right|,\eea where we used that since $\Sigma$ is a $C^{2,\alpha}$ graph on $\mathcal{I}\times \partial D$ (see \eqref{eq:control-zeta}), we have that $d\HH^{d-1}_\Sigma$ is comparable with $d\HH^{d-2}\,dt$.
Then, by a trace inequality on the domains $D_t$, for every $\tau>0$ there exists $C_\tau>0$ such that
\be\label{eq:eq2} \left|\int_{\Sigma}\eta^2q\partial_{\overline\nu} q\,d\HH^{d-1}_\Sigma\right|\le \tau \int_{\mathcal{I}}\int_{D_t}\eta^2|\nabla_\theta q|^2\,d\HH^{d-1}\,dt+C_\tau\int_{\mathcal{I}}\int_{D_t}\eta^2q^2\,d\HH^{d-1}\,dt,\ee
where $C_\tau$ is independent of $t$ since the domains $D_t$ are given by uniformly $C^{2,\alpha}$ graphs over $D$, for every $t\in \mathcal{I}$.
Combining \eqref{eq:eq1} and \eqref{eq:eq2} and choosing $\tau$ small enough, we get
\bea \int_{\mathcal{D}}\eta^2 (|\partial_t q|^2+|\nabla_\theta q|^2)\,d\HH^{d-1}\,dt&\le C\int_{\mathcal{D}} (\eta^2+(\partial_t \eta)^2)q^2\,d\HH^{d-1}\,dt,
\eea
for some $C=C(d,b)>0$.

Now we choose a function $\eta(t):=\widetilde \eta(t-\log\rho)$ where $\widetilde \eta\in C^\infty_c((\log (9/16),\log(15/16)))$ and $\widetilde \eta\equiv1$ in $(\log (5/8),\log (7/8))$. Then $\eta\in C^\infty_c(\mathcal{I})$ and $\eta\equiv1$ in $\mathcal{I}'$. Therefore, the previous inequality concludes the proof. 
\end{proof}
\begin{proof}[Proof of \cref{prop:propfundamental}]
By \cref{proposition:corollary}, and setting $q(t,\theta):=r\partial_r\phi_r(\theta)$ with $t=\log r$, we have for every $t\in \mathcal{I}:=(\log(9\rho/16),\log(15\rho/16))$ \be\label{eq:succ}\|\nabla_{{\rm red}}\mathcal{F}(\phi_{e^t})\|^2\le C\|q(t,\cdot)\|_{L^2(\partial D_t)}^2+C\|\partial_t q(t,\cdot)\|_{L^2( D_t)}^2+C\|q(t,\cdot)\|_{L^2( D_t)}^2,\ee where $D_t:=D_{\zeta(e^t,\cdot)}$.
We notice that the boundary term is controlled by the uniform trace inequality 
$$\|q(t,\cdot)\|_{L^2(\partial D_t)}^2\le C\left(
\|q(t,\cdot)\|_{L^2(D_t)}^2+\|\nabla_\theta q(t,\cdot)\|_{L^2(D_t)}^2\right),$$
where the constant $C$ is independent of $t$, since the domains $D_t$ are given by uniformly $C^{2,\alpha}$ graphs over $D$, for every $t\in \mathcal{I}$.

Integrating \eqref{eq:succ} in $t\in \mathcal{I}':=(\log(5\rho/8),\log(7\rho/8))$ and using the Caccioppoli inequality in \cref{lemma:estimateintime}, we get $$\int_{\mathcal{I}'}\|\nabla_{{\rm red}}\mathcal{F}(\phi_{e^t})\|^2\,dt\le C\int_{\mathcal{I}}\int_{D_t}q^2\,d\HH^{d-1}\,dt \le C\int_{\mathcal{I}}\int_{\partial B_1}q^2\,d\HH^{d-1}\,dt .$$ 
This implies the desired estimate, after the change of variables $t=\log r$, since $q(t,\theta)=\nabla u_r\cdot x-u_r$ on $\partial B_1$.
\end{proof}

\section{Proof of the main results}\label{sec6}
In this section we prove the epiperimetric inequality for variational solutions of the one-phase problem in \cref{thm:epiperimetric-inequality}, and we apply it at every scale to deduce our uniqueness result \cref{thm:uniqueness}. 

The epiperimetric inequality is a consequence of the \L ojasiewicz inequality in \cref{thm:loj-ineq} and the radial control of the reduced gradient in \cref{prop:propfundamental}.
\begin{proof}[Proof of \cref{thm:epiperimetric-inequality}]
    We claim that $\phi_r:=u_r|_{\partial B_1}$ satisfies the hypotheses of the \L ojasiewicz inequality in \cref{thm:loj-ineq} for every $r\in I':=(5\rho/8,7\rho/8)$. 
    
    First, we observe that $W(u_r)\ge \Theta(u,0)\ge W(b)$ by the monotonicity formula \cref{prop:monotonicity-formula}. Moreover, by the equipartition of the energy in \eqref{eq:equipartition-energy}, and using that $W(z_r)=\frac1d\mathcal{F}(\phi_r)$ and $W(b)=\frac1d\mathcal{F}(b)$, we obtain \be\label{eq:eqeq2}\frac1d(\mathcal{F}(\phi_r)-\mathcal{F}(b))=W(u_r)-W(b)+\frac1d\int_{\partial B_1}(\nabla u_r\cdot x-u_r)^2\,d\HH^{d-1}.\ee 
    
    Let $\delta_0$ be the constant in \cref{thm:loj-ineq}. By \cref{lemma:smooth-parametrization-lemma}, for every $r\in I'$, $\partial\Omega_{u_r}\cap \partial B_1$ is given by the graph of a $C^{2,\alpha}$ function with $C^{2,\alpha}$ norm uniformly bounded by $\delta_0$. Moreover, by the $L$-Lipschitz regularity of $u$ in \eqref{eq:lipschitz-regularity} and since $0\in\partial\Omega_u$, the $C^{0,1}$-norm of $u$ is controlled by a constant depending only on $d$ and $L$. Then, by \eqref{eq:eqeq2}, for every $r\in I'$, we deduce the bound $$|\mathcal{F}(\phi_r)-\mathcal{F}(b)|\le C,$$ for some $C=C(d,L)>0$. Additionally, by \eqref{eq:inequality-with-log}, we have, for every $r\in I'$, $$\|u_r-u_\rho\|_{L^2(\partial B_1)}\le C (W(u_\rho)-W(u_r))^{1/2}\le C(W(u_\rho)-W(b))^{1/2}\le C\delta^{1/2}.$$  Therefore, for every $r\in I'$,
    $$\|\phi_r-b\|_{L^2(\partial B_1)}\le \|u_\rho-b\|_{L^2(\partial B_1)}+\|u_r-u_\rho\|_{L^2(\partial B_1)}\le \delta+C\delta^{1/2}.$$ Then, we choose $\delta$ small enough such that $\delta+C\delta^{1/2}\le \delta_0$. Therefore, the hypotheses of \cref{thm:loj-ineq} are satisfied. This concludes the proof of the claim. 

\medskip
    
    Once the claim is proves, we set $$\overline \eps:=\eps E(\rho/2)^{1-2\theta},\quad\text{where}\quad E(\rho):= W(u_\rho)-W(b),$$ $\theta\in(0,\frac12]$ is the constant in \cref{thm:loj-ineq} and $\eps>0$ to be chosen later. Notice that, by the $L$-Lipschitz regularity of $u$ in \eqref{eq:lipschitz-regularity}, $E(\rho/2)$ is bounded by a constant depending only on $d$, $b$ and $L$. Therefore, we can choose $\eps$ sufficiently small so that $\overline \eps\le 1/2$. 
    
    Since $\rho/2\le r$ for every $r\in I'$, then the monotonicity formula \cref{prop:monotonicity-formula} and \eqref{eq:eqeq2} give, for every $r\in I'$ \bea\label{eq:one-identity}0\le E(\rho/2)\le E(r)\le \frac1d(\mathcal{F}(\phi_r)-\mathcal{F}(b)).\eea
    Then, by the \L ojasiewicz inequality in \cref{thm:loj-ineq}, we have, for every $r\in I'$,  $$0\le \overline\eps E(\rho/2)\le C\eps|\mathcal{F}(\phi_r)-\mathcal{F}(b)|^{2-2\theta}\le C\eps\|\nabla_{{\rm red}}\mathcal{F}(\phi_r)\|^2,$$ for some $C=C(d,b,L)>0$. 
    
    Dividing by $r$ the previous inequality and integrating, we obtain by \cref{prop:propfundamental} $$\overline\eps E(\rho/2)\le C\eps\int_{I'}\frac1r\|\nabla_{{\rm red}}\mathcal{F}(\phi_r)\|^2\,dr\le C\eps\int_{\rho/2}^{\rho}\frac1r\int_{\partial B_1}(\nabla u_r\cdot x-u_r)^2\,d\HH^{d-1}\,dr.$$
Integrating the Weiss' monotonicity formula in \eqref{eq:monotonicity-formula},
we also have
    \bea\label{eq:eqeq1}E(\rho)-E(\rho/2)=W(u_{\rho})-W(u_{\rho/2})=\int_{\rho/2}^{\rho}\frac2r\int_{\partial B_1}(\nabla u_r\cdot x-u_r)^2\,d\HH^{d-1}\,dr
    .\eea Then, if we choose $\eps>0$ sufficiently small such that $C\eps\le 1$, the previous inequality implies that $$2\overline\eps E(\rho/2)\le C\eps(E(\rho)-E(\rho/2))\le E(\rho)-E(\rho/2).$$  Rearranging the terms and using that $\overline\eps\le1/2$, we have $$E(\rho/2)\le \frac{1}{1+2\overline \eps}E(\rho)\le (1-\overline \eps)E(\rho),$$ concluding the proof.
\end{proof}

In order to prove the uniqueness result in \cref{thm:uniqueness}, we apply the epiperimetric inequality \cref{thm:epiperimetric-inequality} at every scale.
\begin{proposition}\label{prop:application-epiperimetric}
    Let $u\in H^1(B_1)$ be a variational solution of the one-phase problem. Suppose that $b$ is a blow-up of $u$ at $0\in\partial\Omega_u$, and $b$ is a one-phase cone with isolated singularity. Then, up to a rescaling, $u_\rho$ satisfies the hypotheses of \cref{thm:epiperimetric-inequality} for every $\rho\in(0,1]$.
\end{proposition}
We first prove the following convergence sequence lemma.
\begin{lemma}\label{lemma:sequence-lemma}
    Let $\{e_k\}$ be a sequence of non-increasing real numbers such that $e_k\in[0,1]$ and $e_0>0$. Let $k_0\in\N$ and suppose that \be\label{eq:hyp-sequence}e_{k+1}\le (1-\eps e_{k+1}^\gamma)e_k\quad\text{for every } k=0,\ldots ,k_0, \ee for some $\gamma\in(0,1)$ and $\eps\in(0,1/2)$. Then \be\label{inductive}e_k\le({e_0^{-\gamma}+ck})^{-\frac1\gamma}\quad\text{for every } k=0,\ldots, k_0,\ee for some constant $c>0$ depending only on $\gamma$ and $\eps$. 
\end{lemma}
\begin{proof}
If $e_{\overline k}=0$ for some $\overline k\in\{1,\ldots,k_0\}$, then, by monotonicity of $\{e_k\}$, we have $e_k=0$ and \eqref{inductive} holds for every $k=\overline k,\ldots,k_0$. Then, the argument below applies up to the first vanishing index, and thus, without loss of generality, we can suppose that $e_k\in(0,1]$ for every $k=0,\ldots,k_0$.

Using the trivial inequality $(1-\eps e_{k+1}^\gamma)^{-\gamma}\ge 1+\gamma \eps e_{k+1}^\gamma$
and setting $a_k:=e_k^{-\gamma}$, then \eqref{eq:hyp-sequence} can be rewritten, for every $k=0,\ldots,k_0-1$, as $$a_{k+1}\ge a_{k}(1+\gamma\eps e_{k+1}^\gamma)=a_{k}+\frac{\gamma\eps a_k}{a_{k+1}} $$ Thus $a_{k+1}^2-a_{k}a_{k+1}-\gamma\eps a_k\ge0,$ implying that \bea a_{k+1}&\ge\frac12\left(a_k+\left(a_k^2+4\gamma\eps a_k\right)^{1/2}\right)=\frac12\left(a_k+a_k\left(1+\frac{4\eps\gamma}{a_k}\right)^{1/2}\right)\eea
Using that $(1+t)^{1/2}\ge 1+t/4$ for every $t\in[0,8]$, we have  
\bea a_{k+1}\ge\frac12\left(a_k+a_k\left(1+\frac{\eps\gamma}{a_k}\right)\right)=a_k+\frac{\gamma\eps}2,
\eea where we used that $a_k\ge1$ and $\eps\in(0,1/2)$. Summing in $k$, this implies that $$a_{k}\ge a_0+\frac{\gamma\eps}2k\quad\text{for every } k=0,\ldots, k_0.$$ In terms of $e_k$, this is exactly \eqref{inductive}.
\end{proof}
\begin{proof}[Proof of \cref{prop:application-epiperimetric}] 
    Let $\overline r>0$ and $\delta_1=\delta_1(\overline r)>0$ small to be chosen later.
Since $b$ is a blow-up of $u$ at $0$, then, up to rescaling $u$, we can assume that \bea\label{eq:eq2final}\|u-b\|_{L^2(\partial B_1)}\le \delta_1 \quad\text{and}\quad 0\le W(u)-W(b) \le \delta_1
.\eea
Then, by \eqref{eq:inequality-with-log} applied to $r_1=r$ and $r_2=1$, and using that $W(u_r)\ge W(b)$ by the monotonicity formula \cref{prop:monotonicity-formula}, we obtain, for every $r\in[\overline r/2,1]$,
\be\label{utiledopo-application}\|u_r-b\|_{L^2(\partial B_1)}\le \|u-u_r\|_{L^2(\partial B_1)}+\|u-b\|_{L^2(\partial B_1)}\le C\left( \log\frac2{\overline r}\right)^{1/2}\delta_1^{1/2}+\delta_1\le C({\overline r})\delta_1^{1/2},\ee 
for some $C({\overline r})>0$ depending on $\overline r$. Moreover, for every $r\in (0,1)$ we have $$0\le W(u_r)-W(b)\le \delta_1,$$ again by the monotonicity formula.
We may also assume that $W(u)-W(b)>0$, otherwise the monotonicity formula implies that $u$ is $1$-homogeneous, and hence $u\equiv b$, since $b$ is a blow-up of $u$ at the origin.

Let $\delta>0$ and $\gamma\in[0,1)$ be the constants appearing in \cref{thm:epiperimetric-inequality}. 
We suppose that $\gamma\in(0,1)$. The case $\gamma=0$ follows similarly, with a simpler argument.
We choose $\delta_1=\delta_1(\overline r)>0$ sufficiently small so that $\max\{C(\overline r)\delta_1^{1/2},\delta_1\}\le \delta$. Then, $u$ satisfies the hypotheses of \cref{thm:epiperimetric-inequality} for every $\rho\in[\overline r/2,1]$.

Let $r_0\in[0,\overline r/2)$ be the minimum of the radii $\rho$ such that the hypotheses of \cref{thm:epiperimetric-inequality} are satisfied for every $\rho\in(r_0,1]$.
We claim that $r_0=0$. Suppose, by contradiction, that $r_0>0$. 
We choose $k_0\in\N$ such that $2^{-(k_0+1)}\le r_0 <2^{-k_0}$. Then, since
$r_0 < 2^{-k_0}$, we can apply the epiperimetric inequality in \cref{thm:epiperimetric-inequality} for $\rho=2^{-k}$, for every $k=0,\ldots,k_0$. By \cref{thm:epiperimetric-inequality}, we obtain that the sequence $e_k:=E(2^{-k})=W(u_{2^{-k}})-W(b)$ satisfies the hypotheses of \cref{lemma:sequence-lemma}, for $k=0,\ldots,k_0$, and thus $$e_k\le Ck^{-\frac1\gamma}\quad\text{for every }k=1,\ldots,k_0.$$
Therefore, for every $r\in(r_0,1/2)$, we choose $k\in\{1,\ldots,k_0\}$ such that $2^{-(k+1)}< r\le 2^{-k}$. Then, by the previous estimate and the monotonicity formula \cref{prop:monotonicity-formula}, we obtain the following decay of the Weiss' energy $$0\le W(u_r)-W(b)\le e_k\le Ck^{-\frac1\gamma}\le C|\log r|^{-\frac1\gamma},\quad\text{for every }r\in(r_0,1/2).$$ This decay and a standard dyadic argument (see e.g.~\cite[Proposition 5.1]{csv18}) give, for some $\alpha\in(0,1)$, \be\label{eq:application-epi}\|u_r-u_{\overline r}\|_{L^2(\partial B_1)}\le {C}|\log \overline r|^{-\alpha}\quad\text{for every }r\in [r_{0},\overline r).\ee 
Thus, by \eqref{eq:inequality-with-log}, \eqref{utiledopo-application} and  \eqref{eq:application-epi}, for every $r'\in(r_{0}/2,r_{0}]$, we have \bea \|u_{r'}-b\|_{L^2(\partial B_1)}&\le \|u_{r'}-u_{r_{0}}\|_{L^2(\partial B_1)}+\|u_{r_{0}}-u_{\overline r}\|_{L^2(\partial B_1)}+\|u_{\overline r}-b\|_{L^2(\partial B_1)}\\&\le C\log(2)^\frac12\delta_1^{1/2}+C|\log\overline r|^{-\alpha}+C(\overline r)\delta_1^{1/2}.\eea 
Choosing first $\overline r>0$ sufficiently small and then $\delta_1=\delta_1(\overline r)>0$ small enough, we can ensure that the right-hand side above is bounded by $\delta$. Therefore, $u$ satisfies the hypotheses of \cref{thm:epiperimetric-inequality} for $\rho\in(r_0/2,r_0]$, contradicting the definition of $r_0$.
\end{proof}
\begin{proof}[Proof of \cref{thm:uniqueness}] By \cref{prop:application-epiperimetric}, the epiperimetric inequality in \cref{thm:epiperimetric-inequality} can
be applied at every scale. Then, the uniqueness of the blow-up limit and the corresponding $L^2(\partial B_1)$ logarithmic rate of convergence follow as in \eqref{eq:application-epi}. The $L^\infty$ rate of convergence is a consequence of an interpolation inequality (see e.g. \cite[Equation 9.1]{carduccitortoneuniqueness}). 
Moreover, arguing as in the proof of \cite[Theorem 1]{esv}, we obtain that $\partial\Omega_u\cap B_{r_0}$ is a $C^{1,\log}$ graph over $\partial\Omega_b\cap B_{r_0}$. 
Finally, if $b$ is integrable, we can use the epiperimetric inequality with $\gamma=0$ to improve the logarithmic convergence to $r^\alpha$, and then the regularity of the graph to $C^{1,\alpha}$. 
\end{proof}
\bibliographystyle{alpha}
\bibliography{FreeBoundary_bib}
\end{document}